\documentclass[oneside,10pt]{article}
\usepackage[b5paper]{geometry}	

\usepackage{amsfonts,amsmath,latexsym,amssymb}
\usepackage{mathrsfs,upref}
\usepackage{mathptmx}		    	                            		
\usepackage{amscd}
\usepackage{amsthm}
\usepackage{graphicx}
\usepackage{amsmath}
\usepackage{lipsum}

\newcommand\blfootnote[1]{%
  \begingroup
  \renewcommand\thefootnote{}\footnote{#1}%
  \addtocounter{footnote}{-1}%
  \endgroup}

\newtheorem{theorem}{Theorem}

\newtheorem{corollary}[theorem]{Corollary}

\newtheorem{definition}[theorem]{Definition}
\newtheorem{example}[theorem]{Example}

\newtheorem{lemma}[theorem]{Lemma}

\newtheorem{proposition}[theorem]{Proposition}
\newtheorem{remark}[theorem]{Remark}

\newcommand\supp{\mathop{\rm supp}}
\newcommand\esssup{\mathop{\rm ess \, sup}}
\newcommand\essinf{\mathop{\rm ess \, inf}}

\begin{document}

\title{Estimates for convolution operators on Hardy spaces associated with ball quasi-Banach function spaces}
\author{Pablo Rocha}

\maketitle

\begin{abstract}
Let $0 \leq \alpha < n$, $N \in \mathbb{N}$, and let $X$ and $Y$ be ball quasi-Banach function spaces on $\mathbb{R}^n$. We consider 
operators $T_{\alpha}$ defined by convolution with kernels of type $(\alpha, N)$. Assuming that the powered Hardy-Littlewood maximal 
operator satisfies a Fefferman-Stein vector-valued maximal inequality on $X$ and is bounded on the associated space, we prove that, 
for $\alpha = 0$, $T_0$ extends to a bounded operator $H_{X}(\mathbb{R}^n) \to X$ and $H_{X}(\mathbb{R}^n) \to H_{X}(\mathbb{R}^n)$; and for 
$0 < \alpha < n$, under certain additional assumptions on $X$ and $Y$, $T_{\alpha}$ extends to a bounded operator $H_{X}(\mathbb{R}^n) \to Y$ and $H_{X}(\mathbb{R}^n) \to H_{Y}(\mathbb{R}^n)$. In particular, these results yield the corresponding estimates for singular integral operators and Riesz potentials. We also provide an off-diagonal Fefferman-Stein vector-valued inequality for the fractional maximal operator on the $p$-convexification of ball quasi-Banach function spaces.
\end{abstract}


\section{Introduction}

The concept of a ball quasi-Banach function space was introduced by Y. Sawano, K.-P. Ho, D. Yang and S. Yang in \cite{sawa} to unify, under 
a broader common framework than the one given by (quasi-)Banach function spaces, many variants of classical Hardy-type spaces such as 
weighted Hardy spaces, Hardy-Lorentz spaces, Hardy-Orlicz spaces, Hardy-Herz spaces, Hardy-Morrey spaces, Musielak-Orlicz-Hardy spaces, 
and variable Hardy spaces, among others (see \cite{sawa} and the references therein). For instance, if we consider the space 
$L^{p}(\mathbb{R}^n)$, $0 < p < \infty$, with the usual quasi-norm given by $\| f \|_{p}^p = \int_{\mathbb{R}^n} |f(x)|^p dx$, then 
$X:=(L^{p}(\mathbb{R}^n), \| \cdot \|_{p})$ is a ball quasi-Banach function space and the related Hardy type space 
$H_X (\mathbb{R}^n)$ coincides with $H^p(\mathbb{R}^n)$, where $H_X (\mathbb{R}^n)$ is defined by (\ref{Hx def}) below, and 
$H^p(\mathbb{R}^n)$ is the classical Hardy space defined in \cite{Stein}.

Two of the principal results obtained in \cite{sawa} are the atomic and molecular characterizations of the Hardy space $H_X (\mathbb{R}^n)$ associated with a ball quasi-Banach function space $X$. These characterizations rely on the Fefferman-Stein vector-valued maximal inequality of the powered Hardy-Littlewood maximal operator on $X$ and its boundedness on the associate space $X'$ (see (\ref{A1}) and (\ref{A2}) below). As is well known, in the classical setting, such decompositions are very useful when one studies the behavior of certain operators, such as singular and fractional integral operators, on a given Hardy-type space (see for instance \cite{cuerva}, \cite{Nakai}, \cite{Ro-Urc}, 
\cite{Stein}, \cite{wheeden}, \cite{Taible}).

In \cite{Wang}, F. Wang et al. established a characterization of $H_X (\mathbb{R}^n)$ in terms of Littlewood-Paley functions, and obtained the boundedness of $\delta$-type Calder\'on-Zygmund operators on $H_X (\mathbb{R}^n)$. They also studied local Hardy spaces $h_X(\mathbb{R}^n)$ associated with a ball Banach function space $X$. After that, D.-C. Chang et al. in \cite{Chang}, under some weak assumptions on the Littlewood-Paley functions, improved the existing characterizations in terms of Littlewood-Paley functions of $H_X (\mathbb{R}^n)$. Their results have applications in Morrey spaces, mixed-norm Lebesgue spaces, variable Lebesgue spaces, weighted Lebesgue spaces and Orlicz-slice spaces.

K.-P. Ho in \cite{Ho}, by means of extrapolation techniques, obtained mapping properties on Hardy spaces associated with ball quasi-Banach function spaces for strongly singular Calder\'on-Zygmund operators with applications to Hardy Orlicz-slice spaces, variable Hardy local Morrey spaces and variable Herz-Hardy spaces.

A finite atomic characterization of $H_X(\mathbb{R}^n)$ was given by X. Yan, D. Yang and W. Yuan in \cite{Yan}. As an
application, they prove that the dual space of $H_X(\mathbb{R}^n)$ is the Campanato space associated with $X$.
Recently, Y. Tan in \cite{Tan}, by means of this finite atomic characterization, proved the boundedness of multilinear fractional integral
operators from products of Hardy spaces associated with ball quasi-Banach function spaces into another ball quasi-Banach function space, which are related to each other. In \cite{Procha}, the present author pointed out that operators of this kind are not bounded from a product of Hardy spaces into a Hardy space.

Z. Nieraeth in \cite{Nieraeth}, assuming the boundedness of Hardy-Littlewood maximal operator $M$ on $[(X')^{\frac{1}{1-\frac{\alpha}{n}}}]'$ and $(X')^{\frac{1}{1-\frac{\alpha}{n}}}$, proved the $X \to ([(X')^{\frac{1}{1-\frac{\alpha}{n}}}]')^{1-\frac{\alpha}{n}}$ boundedness for the Riesz potential $I_{\alpha}$, $0 < \alpha < n$, where $X$ is a $\frac{n}{\alpha}$-concave Banach function space over $\mathbb{R}^n$. 
Later, Y. Chen, H. Jia and D. Yang in \cite{YChen}, assuming (\ref{A1}) and (\ref{A2}) below, proved that the Riesz potential 
$I_{\alpha}$ can be extended to a bounded operator from $H_X(\mathbb{R}^n)$ to $H_{X^{\beta}}(\mathbb{R}^n)$, $\beta > 1$, if and only if 
for any ball $B \subset \mathbb{R}^n$, $|B|^{\frac{\alpha}{n}} \lesssim \| \chi_B \|_{X}^{(\beta-1)/\beta}$, where $X$ is a ball quasi-Banach function space and $X^{\beta}$ denotes the $\beta$-convexification of $X$. Moreover, using extrapolation techniques, the authors also proved the $H_X(\mathbb{R}^n) \to H_Y(\mathbb{R}^n)$ boundedness of $I_{\alpha}$, for certain ball quasi-Banach function spaces $X$ and $Y$.

The first main result of this paper concerns the \textit{fractional maximal operator} $M_{\alpha}$, which is defined by
\[
(M_{\alpha}f)(x) = \sup_{B \ni x} |B|^{\frac{\alpha}{n} - 1}\int_{B} |f(y)| \, dy,
\]
where $0 < \alpha < n$, $f$ is a locally integrable function on $\mathbb{R}^{n}$ and the supremum is taken over all balls $B$ containing 
$x$. This result is stated in Proposition \ref{vector-valued fract} below and provides an off-diagonal Fefferman-Stein vector-valued inequality for $M_{\alpha}$ on the $p$-convexification of ball quasi-Banach function spaces. This inequality is useful to study convolution operators with kernels of type $(\alpha, N)$ on $\mathbb{R}^{n}$, a concept we now proceed to define.

Let $0 \leq \alpha < n$ and $N \in \mathbb{N}$. For $0 < \alpha < n$, a function $K_{\alpha} \in C^{N}(\mathbb{R}^{n} \setminus \{ 0 \})$ is said to be a kernel of type $(\alpha, N)$ on $\mathbb{R}^{n}$ if
\begin{equation} \label{decay0}
\left|(\partial^{\beta}K_{\alpha})(x) \right| \lesssim |x|^{\alpha - n - |\beta|} \,\,\,\, \text{for all} \,\, |\beta| \leq N \,\,\, \text{and all} \,\, x \neq 0,
\end{equation}
where $\partial^{\beta}$ is the higher order partial derivative associated to the multiindex $\beta = (\beta_1, ..., \beta_n)$, and 
$|\beta| = \beta_1 + \cdot \cdot \cdot + \beta_{n}$. A distribution $K_{0}$ is said to be a kernel of type $(0,N)$ on $\mathbb{R}^{n}$ 
if it is of class $C^{N}$ on $\mathbb{R}^{n} \setminus \{ 0 \}$, satisfies (\ref{decay0}) with $\alpha = 0$, and 
$\| K_0 \ast f \|_{2} \leq C \| f \|_2$ for all $f \in \mathcal{S}(\mathbb{R}^{n})$. Kernels of this type were studied by G. Folland and 
E. Stein in \cite{Folland} in the setting of homogeneous groups.

For $0 \leq \alpha < n$, let $T_{\alpha}$ be the convolution operator defined, say on $\mathcal{S}(\mathbb{R}^n)$, by 
$T_{\alpha}f = K_{\alpha} \ast f$, where $K_{\alpha}$ is a kernel of type $(\alpha, N)$. These operators include the classical singular and fractional integrals (see Sections \ref{Sing integ} and \ref{Fract integ} below). The main purpose of this work is to generalize the results obtained in \cite[Proposition 6.2 and Theorem 6.10]{Folland} to the Hardy spaces associated with ball quasi-Banach function spaces for the case $G = \mathbb{R}^n$. Our main result is Theorem \ref{main thm} in Section \ref{resultados principales}, which states that if $X$ and $Y$ are ball quasi-Banach function spaces satisfying certain hypotheses and $N$ is suitably chosen, then the operator $T_{\alpha}$ satisfies the following mapping properties:
for $\alpha = 0$,
\begin{equation} \label{est 1}
T_0 : H_{X}(\mathbb{R}^n) \to X
\end{equation}
and
\begin{equation} \label{est 2}
T_0 : H_{X}(\mathbb{R}^n) \to H_{X}(\mathbb{R}^n),
\end{equation}
for $0 < \alpha < n$,
\begin{equation} \label{est 3}
T_{\alpha} : H_{X}(\mathbb{R}^n) \to Y
\end{equation}
and
\begin{equation} \label{est 4}
T_{\alpha} : H_{X}(\mathbb{R}^n) \to H_{Y}(\mathbb{R}^n).
\end{equation}
We observe that the kernels of type $(0, N)$ considered here satisfy derivative estimates up to a positive order, and are therefore 
smoother than the $\delta$-type Calder\'on-Zygmund kernels treated in \cite{Wang}. Consequently, estimate (\ref{est 2}) is already covered by 
\cite[Theorem 3.5]{Wang}. Moreover, for certain choices of the parameters, the estimate (\ref{est 1}) is independent of (\ref{est 2}), and the estimate (\ref{est 3}) is independent of (\ref{est 4}) (see Remark \ref{note thm} below). In particular, (\ref{est 3}) and (\ref{est 4}) give estimates for fractional integrals.

\ 

To establish the results of Section \ref{resultados principales}, under certain additional assumptions on $X$ (see Section \ref{hypotheses}), we will use the infinite atomic decomposition and the maximal characterization of $H_X(\mathbb{R}^n)$ established in \cite{sawa} together with a density argument and some vector-valued inequalities of Section \ref{auxiliar}. Among them is the off-diagonal Fefferman-Stein vector-valued inequality for the fractional maximal operator on the $p$-convexification of ball quasi-Banach function spaces.

\

The paper is organized as follows. Section \ref{basico} recalls the basic theory of Hardy spaces associated with ball quasi-Banach function spaces. Some vector-valued maximal inequalities are established in Section \ref{auxiliar}. In Section \ref{resultados principales}, we state and prove our main results. Finally, Section \ref{4 examples} presents four concrete examples to illustrate our results.

\

\textbf{Notation:} The symbol $A \lesssim B$ stands for the inequality $A \leq cB$ for some constant $c$. We denote by $Q(x_0, r)$ the cube centered at $x_0 \in \mathbb{R}^{n}$ with side length $r$. Given $\gamma >0$ and a cube $Q = Q(x_0, r)$, we set $\gamma Q = Q(x_0, \gamma r)$. We denote by $\mathcal{Q}$ the set of all cubes having their edges parallel to the coordinate axes. For a measurable subset 
$E \subset \mathbb{R}^{n}$ we denote by $|E|$ and $\chi_E$ the Lebesgue measure and the characteristic function of $E$, respectively. 
Given a real number $s \geq 0$, we write $\lfloor s \rfloor$ for the floor of $s$. As usual we denote by 
$\mathcal{S}(\mathbb{R}^{n})$ the space of smooth and rapidly decreasing functions, with 
$\mathcal{S}'(\mathbb{R}^{n})$  the dual space. If $\beta$ is the multiindex $\beta=(\beta_1, ..., \beta_n)$, then 
$|\beta| = \beta_1 + ... + \beta_n$. Given a function $g$ on $\mathbb{R}^n$ and $t > 0$, we write $g_t(x) = t^{-n} g(t^{-1} x)$. Let $d$ be a nonnegative integer and $\mathcal{P}_{d}$ the subspace of $L^{1}_{loc}(\mathbb{R}^{n})$ formed by all the polynomials of degree at most $d$. Given a measurable function $h$, the notation $h \perp \mathcal{P}_{d}$ means that $\int h(x) P(x) dx = 0$ for all $P \in \mathcal{P}_{d}$.

Throughout this paper, $C$ will denote a positive constant, not necessarily the same at each occurrence.

\section{Preliminaries} \label{basico}

\subsection{Ball quasi-Banach function spaces} In the sequel, $\mathfrak{M} = \mathfrak{M}(\mathbb{R}^n)$ is the set of all measurable functions on $\mathbb{R}^n$ and $\mathfrak{M}_{+} = \mathfrak{M}_{+}(\mathbb{R}^n)$ is the cone of all non-negative measurable functions on $\mathbb{R}^n$.

For every $x \in \mathbb{R}^n$ and $r > 0$ fixed, let $B(x,r) := \{ y \in \mathbb{R}^n: |x-y| < r \}$. Now, we define
\[
\mathbb{B} := \{ B(x,r) : x \in \mathbb{R}^n \,\, \text{and} \,\, r > 0 \}.
\]

\begin{definition}
A mapping $\rho : \mathfrak{M}_{+} \to [0, \infty]$ is called a ball quasi-Banach function norm if it satisfies the following properties:

$(P1)$ $\rho(f) = 0$ implies that $f = 0$ a.e.; 

$(P2)$ $\rho(\alpha f) = |\alpha| \rho(f)$ for all $\alpha \in \mathbb{C}$ and all $f \in \mathfrak{M}_{+}$;

$(P3)$ there exists $C \geq 1$ such that $\rho(f+g) \leq C (\rho(f) + \rho(g))$ for all $f, g \in \mathfrak{M}_{+}$;

$(P4)$ if some $f, g \in \mathfrak{M}_{+}$ satisfy $f \leq g$ a.e., then $\rho(f) \leq \rho(g)$;

$(P5)$ if some $f_n, f \in \mathfrak{M}_{+}$ satisfy $f_n \uparrow f$ a.e., then $\rho(f_n) \uparrow \rho(f)$;

$(P6)$ if $B \in \mathbb{B}$, then $\rho(\chi_B) < \infty$.
\end{definition}

\begin{definition}
Let $\rho$ be a ball quasi-Banach function norm. We then define the corresponding ball quasi-Banach function space $X = X(\rho)$ as the set
\[
X = \{ f \in \mathfrak{M} : \rho(|f|) < \infty \}.
\]
For each $f \in X$, define
\[
\Vert f \Vert_X = \rho(|f|).
\]
\end{definition}

\begin{remark} \label{prop X}
Let $(X, \Vert \cdot \Vert_X)$ be a ball quasi-Banach function space. Then it is easy to check that

$(i)$ $\Vert f \Vert_X = 0$ implies that $f = 0$ a.e.;

$(ii)$ $\Vert \alpha f \Vert_X = |\alpha| \Vert f \Vert_X$ for all $\alpha \in \mathbb{C}$ and all $f \in X$;

$(iii)$ there exists $C \geq 1$ such that $\Vert f+g \Vert_X \leq C (\Vert f \Vert_X + \Vert g \Vert_X)$ for all $f, g \in X$;

$(iv)$ if $f \in \mathfrak{M}$, $g \in X$ are such that $|f| \leq |g|$ a.e., then $f \in X$ and $\Vert f \Vert_X \leq \Vert g \Vert_X$;

$(v)$ if $0 \leq f_n \uparrow f$ a.e., then either $f \notin X$ and $\Vert f_n \Vert_X \uparrow \infty$, or $f \in X$ and 
$\Vert f_n \Vert_X \uparrow \Vert f \Vert_X$;

$(vi)$ if $B \in \mathbb{B}$, then $\chi_B \in X$. \\
Moreover, from $(P4)$ and $(P6)$ it follows that $\chi_Q \in X$, for each cube $Q \subset \mathbb{R}^n$.
\end{remark}

A ball quasi-Banach function space $X$ is called a \textit{ball Banach function space} if the constant $C$ in Remark \ref{prop X}$(iii)$ can be chosen equal to $1$, and for any ball $B \in \mathbb{B}$ there exists a positive constant $C_{(B)}$, depending on $B$, such that
\[
\int_{B} |f(x)| dx \leq C_{(B)} \| f \|_{X},
\]
for all $f \in X$. Thus, every Banach function space is a ball Banach function space (see \cite[Definition 1.3]{Bennett}).

A discussion of the conditions under which (quasi-)Banach function spaces are well defined was given by E. Lorist and Z. Nieraeth in \cite{Lorist} (see also \cite{Nekvinda}).

\begin{definition}
For any ball Banach function space $X$, the associate space (also called the K\"othe dual) $X'$ is defined by setting
\[
X' := \left\{ f \in \mathfrak{M} : \| f \|_{X'} := \sup \{ \| f g\|_{L^{1}(\mathbb{R}^n)} : g \in X, \| g \|_{X} = 1 \} < \infty \right\}
\]
where $\| \cdot \|_{X'}$ is called the associate norm of $\| \cdot \|_{X}$.
\end{definition}

\begin{lemma} (\cite[Proposition 2.3]{sawa} and \cite[Lemma 2.6]{Zhang}) \label{doble dual}
Let $X$ be a ball Banach function space. Then $X'$ is also a ball Banach function space and $X$ coincides with its second associate space $X''$. In other words, a function $f$ belongs to $X$ if and only if it belongs to $X''$ and, in that case,
\[
\| f \|_{X} = \| f \|_{X''}.
\]
\end{lemma}

\begin{definition}
A ball quasi-Banach function space $X$ is said to have an absolutely continuous quasi-norm if, for any $f \in X$,  $\| f \chi_{E_j} \|_{X} \downarrow 0$ whenever $\{ E_j \}_{j=1}^{\infty}$ is a sequence of measurable sets that satisfy $E_j \supset E_{j+1}$ for all $j \in \mathbb{N}$ and 
$\bigcap_{j=1}^{\infty} E_j = \emptyset$.
\end{definition}

\begin{definition}
Let $X$ be a ball quasi-Banach function space and $p \in (0, \infty)$. The $p$-convexification $X^p$ of $X$ is defined by setting 
$X^p := \{ f \in \mathfrak{M} : |f|^p \in X \}$ equipped with the quasi-norm $\| f \|_{X^p} : =  \| |f|^p \|^{1/p}_{X}$.
\end{definition}

\begin{definition}
Let $X$ be a ball quasi-Banach function space and $p \in (0, \infty)$. The space $X$ is said to be $p$-convex if there exists a positive constant $C$ such that, for any $\{ f_j \}_{j=1}^{\infty} \subset X^{1/p}$,
\[
\left\| \sum_{j=1}^{\infty} |f_j| \right\|_{X^{1/p}} \leq C \sum_{j=1}^{\infty} \left\| f_j \right\|_{X^{1/p}}.
\]
In particular, when $C = 1$, $X$ is said to be strictly $p$-convex.
\end{definition}

\subsection{Maximal functions and additional assumptions} \label{hypotheses} For $0 \leq \alpha < n$, the \textit{fractional maximal operator} 
$M_{\alpha}$ is given by
\[
(M_{\alpha}f)(x) = \sup_{B \ni x} |B|^{\frac{\alpha}{n} - 1}\int_{B} |f(y)| \, dy,
\]
where $f$ is a locally integrable function on $\mathbb{R}^{n}$ and the supremum is taken over all balls $B$ containing $x$. For 
$\alpha=0$, we have that $M_0 = M$, where $M$ is the {\it Hardy-Littlewood maximal operator} on $\mathbb{R}^{n}$.

For $\theta \in (0, \infty)$, the \textit{powered Hardy-Littlewood maximal operator} $M^{(\theta)}$ is defined by
\[
(M^{(\theta)}f)(x) = \left[ M (|f|^{\theta})(x) \right]^{1/ \theta}.
\]
Throughout the sequel, we impose the following two additional assumptions:

$A1)$ Let $X$ be a ball quasi-Banach function space. Assume that, for some $\theta, s \in (0, 1]$ with $\theta < s$, there exists a positive constant $C$ such that for any sequence of functions $\{ f_j \}_{j=1}^{\infty} \subset L^{1}_{loc}(\mathbb{R}^n)$
\begin{equation} \label{A1}
\left\Vert  \left\{ \sum_{j=1}^{\infty} (M^{(\theta)} f_j)^{s} \right\}^{1/s} \right\Vert_{X} \leq C
\left\Vert \left\{ \sum_{j=1}^{\infty} |f_j|^{s} \right\}^{1/s}  \right\Vert_{X}.
\end{equation}

$A2)$ Let $X$ be a ball quasi-Banach function space. Assume that there exist $r \in (0, 1]$, $q \in (1, \infty)$ and a positive constant $C$ such that $X^{1/r}$ is a ball Banach function space and for any $f \in (X^{1/r})'$
\begin{equation} \label{A2}
\left\Vert  M^{((q/r)')} f \right\Vert_{(X^{1/r})'} \leq C \left\Vert f  \right\Vert_{(X^{1/r})'}.
\end{equation}

\begin{remark} \label{cond equiv A2}
We observe that (\ref{A2}) is equivalent to the boundedness of the Hardy-Littlewood maximal operator on $[(X^{1/r})']^{1/(q/r)'}$.
\end{remark}

Given $L \in \mathbb{Z}_{+}$, let 
\[
\mathcal{F}_{L}=\left\{ \varphi \in \mathcal{S}(\mathbb{R}^{n}):\sum\limits_{\left\vert \mathbf{\beta }\right\vert \leq L}\sup\limits_{x\in \mathbb{R}^{n}}\left( 1+\left\vert x\right\vert \right)^{L}\left\vert \partial^{\mathbf{\beta }}
\varphi(x) \right\vert := \| \varphi \|_{\mathcal{S}(\mathbb{R}^{n}), \, L}  \leq 1\right\}.
\] 
Let $f \in \mathcal{S}'(\mathbb{R}^{n})$, $b \in (0, \infty)$, $\Phi \in \mathcal{S}(\mathbb{R}^n)$ and $\mathbb{R}^{n+1}_{+} := \mathbb{R}^n 
\times (0, \infty)$. We define the following two maximal functions for $f$, 
\begin{equation} \label{grand max cero}
\mathcal{M}^{0}_{L} f(x) :=\sup \left\{ |\left( \phi_t \ast f\right)(x) | : t > 0, \phi \in \mathcal{F}_{L} \right\},
\end{equation}
and
\begin{equation} \label{max b}
M^{**}_{b}(f, \Phi)(x) :=\sup_{(y, t) \in \mathbb{R}^{n+1}_{+}} \frac{ |( \Phi_t \ast f )(x-y)|}{(1 + t^{-1}|y|)^b}.
\end{equation}

\subsection{Hardy spaces associated with ball quasi-Banach function spaces} Let $(X, \Vert \cdot \Vert_X)$ be a ball quasi-Banach function space. Following \cite{sawa}, we introduce the Hardy-type space associated with $X$, which is denoted by $H_{X}(\mathbb{R}^n)$, and present the atomic decomposition for elements of $H_{X}(\mathbb{R}^n)$ also established in \cite{sawa}.

\begin{definition} 
Let $X$ be a ball quasi-Banach function space. Then the Hardy space $H_{X}(\mathbb{R}^n)$ associated with $X$ is defined as
\begin{equation}  \label{Hx def}
H_{X}(\mathbb{R}^n) := \left\{ f \in \mathcal{S}'(\mathbb{R}^n) : \| f \|_{H_{X}(\mathbb{R}^n)} := \| M^{**}_{b}(f, \Phi) \|_{X} 
< \infty \right\},
\end{equation}
where $M^{**}_{b}$ is the maximal operator given by (\ref{max b}) with $b$ sufficiently large.
\end{definition}

\begin{theorem} (\cite[Theorem 3.1 - (ii)]{sawa}) \label{equiv grand max} 
Let $X$ be a ball quasi-Banach function space such that the Hardy-Littlewood maximal operator $M$ is bounded on $X^{1/r}$ for some 
$r \in (0, \infty)$, and let $\mathcal{M}^{0}_{L} f$ be the grand maximal of $f$ given by (\ref{grand max cero}). Assume that
$b \in (n/r, \infty)$. Then, when $L \geq \lfloor b+2 \rfloor$, if one of the quantities 
\[
\| M^{**}_{b}(f, \Phi) \|_{X} \,\,\,\, \text{or} \,\,\,\, \| \mathcal{M}^{0}_{L} f \|_{X}
\]
is finite, then the other is also finite, and the two quantities are equivalent with the implicit positive constants independent of $f$.
\end{theorem}

\begin{remark} \label{Hx space}
If $X$ and $r \in (0, \infty)$ are as in Theorem \ref{equiv grand max} and $b = n/r + 1$, then for any $L \geq \lfloor  n/r + 3 \rfloor$, we have
\[
H_{X}(\mathbb{R}^n) = \left\{ f \in \mathcal{S}'(\mathbb{R}^n) : \| \mathcal{M}^{0}_{L} f \|_{X} < \infty \right\}.
\]
For a fixed $L \geq \lfloor  n/r + 3 \rfloor$, we use $\| f \|_{H_{X}(\mathbb{R}^n)} = \| \mathcal{M}^{0}_{L} f \|_{X}$.
\end{remark}

Before establishing the atomic decomposition for elements of $H_{X}(\mathbb{R}^n)$, we recall the definition of atoms.

\begin{definition}
Let $X$ be a ball quasi-Banach function space satisfying (\ref{A1}) and let $p \in [1, \infty]$. Assume that $d \in \mathbb{Z}_{+}$ satisfies 
$d \geq d_X$, where $d_X := \lfloor n (1/\theta - 1) \rfloor$ and $\theta \in (0, 1]$ is the constant in (\ref{A1}). Then the function 
$a(\cdot)$ is called an $(X, p, d)$-atom if there exists a cube $Q \in \mathcal{Q}$ such that $\supp(a) \subset Q$,
\[
\| a \|_{L^p(\mathbb{R}^n)} \leq \frac{|Q|^{1/p}}{\| \chi_Q \|_{X}},
\]
and $a(\cdot) \perp \mathcal{P}_d$.
\end{definition}

\begin{remark} \label{infinite atom}
For $p \geq 1$ fixed, every $(X, \infty, d)$-atom is an $(X, p, d)$-atom.
\end{remark}

To prove our main results, we need the following atomic decomposition for elements of $H_{X}(\mathbb{R}^n) \cap L^p(\mathbb{R}^n)$.

\begin{theorem} \label{X atomic decomp}
Let $X$ be a ball quasi-Banach function space satisfying (\ref{A1}) for some $\theta, s \in (0,1]$, $d \geq d_X$ be a fixed integer, 
and $f \in H_{X}(\mathbb{R}^n) \cap L^p(\mathbb{R}^n)$, with $p \in (1, \infty)$ fixed. Then there exist two sequences 
$\{ a_j \}_{j=1}^{\infty}$ of $(X, \infty, d)$-atoms, supported on the cubes $\{ Q_j \}_{j=1}^{\infty} \subset \mathcal{Q}$, 
and $\{ \lambda_j \}_{j=1}^{\infty} \subset (0, \infty)$ such that
\begin{equation} \label{converg Lp}
f = \sum_{j=1}^{\infty} \lambda_j a_j \,\,\,\,\,\,\,\, \text{in} \,\,\,\, L^p(\mathbb{R}^n),
\end{equation}
and
\begin{equation} \label{atomic norm}
\left\| \left\{ \sum_{j=1}^{\infty} \left( \frac{\lambda_j}{\| \chi_{Q_j} \|_X} \right)^s \chi_{Q_j} \right\}^{1/s} \right\|_{X} \lesssim_{s}  \| f \|_{H_{X}(\mathbb{R}^n)},
\end{equation}
where the implicit positive constant is independent of $f$, but depends on $s$.
\end{theorem}

\begin{proof}
The existence of such an atomic decomposition as well as the validity of inequality (\ref{atomic norm}) are guaranteed by 
\cite[Theorem 3.7]{sawa}. The convergence in (\ref{converg Lp}) is initially understood in $\mathcal{S}'(\mathbb{R}^n)$. To see the 
convergence of the atomic series to $f$ in $L^{p}(\mathbb{R}^{n})$, we point out that the construction of such an atomic decomposition 
(see \cite[Proposition 4.3]{sawa}) is analogous to the one given for classical Hardy spaces (see \cite[Chapter III]{Stein}). Since 
$f \in L^{p}(\mathbb{R}^{n})$ we have $\mathcal{M}_{L}^{0} f \in L^{p}(\mathbb{R}^{n})$. So, following the proof of 
\cite[Theorem 3.1]{Rocha}, we obtain (\ref{converg Lp}).
\end{proof}

From \cite[Corollary 3.11]{sawa} and since \cite[Remark 3.12]{sawa} also holds true for any $p \in (1, \infty)$, we have the following result.

\begin{proposition} \label{dense set}
Assume that $X$ is a ball quasi-Banach function space satisfying (\ref{A1}), (\ref{A2}), which is strictly $s$-convex, 
where $s \in (0, 1]$ is as in (\ref{A1}), and $X$ has an absolutely continuous quasi-norm. Then, for any $p \in (1, \infty]$, 
$H_{X}(\mathbb{R}^n) \cap L^p(\mathbb{R}^n)$ is dense in $H_{X}(\mathbb{R}^n)$.
\end{proposition}

\section{Auxiliary results} \label{auxiliar}

The following two lemmas generalize the vector-valued inequalities established in Lemmas 4.9 and 4.11 of \cite{Uribe} to the setting of ball quasi-Banach function spaces. 

\begin{lemma} (\cite[Lemma 2.8 and Remark 2.9]{Chen}) \label{q estimate}
Let $X$ be a ball quasi-Banach function space. Let $r \in (0, 1]$ such that $X^{1/r}$ is a ball Banach function space and $X$ satisfies 
(\ref{A2}) for some $q \geq 1$, then for any countable collection of cubes $\{ Q_j \}$ and non-negative functions $g_j$ such that 
$\supp(g_j) \subset Q_j$ and $\sum_{j=1}^{\infty} \left( \frac{1}{|Q_j|}\int_{Q_j} g_j^{q} \right)^{1/q} \chi_{Q_j} \in X$,
\begin{equation} \label{vector ineq}
\left\Vert  \sum_{j=1}^{\infty}  g_j \right\Vert_{X} \lesssim 
\left\Vert \sum_{j=1}^{\infty} \left( \frac{1}{|Q_j|}\int_{Q_j} g_j^{q} \right)^{1/q} \chi_{Q_j}  \right\Vert_{X}.
\end{equation}
\end{lemma}

\begin{lemma} (\cite[Lemma 2.4]{Tan}) \label{Qalfan}
Let $0 < \alpha < n$ and $0 < p_0 < q_0 \leq 1$ be such that $\frac{1}{p_0} - \frac{1}{q_0} = \frac{\alpha}{n}$, and let $X$ and $Y$ be ball quasi-Banach function spaces such that $X^{1/p_0}$ and $Y^{1/q_0}$ are ball Banach function spaces and 
$(Y^{1/q_0})' = ((X^{1/p_0})')^{p_0/q_0}$. If the Hardy-Littlewood maximal operator $M$ is bounded on $(Y^{1/q_0})'$, then for any countable collection of cubes $\{ Q_j \}$, with $\sum_{j=1}^{\infty} \lambda_j \chi_{Q_j} \in X$, and $\lambda_j > 0$,
\begin{equation} 
\left\Vert  \sum_{j=1}^{\infty} \lambda_j |Q_j|^{\frac{\alpha}{n}} \chi_{2 Q_j} \right\Vert_{Y} \lesssim
\left\Vert \sum_{j=1}^{\infty} \lambda_j \chi_{Q_j}  \right\Vert_{X}.
\end{equation}
\end{lemma}

We conclude this section with two Fefferman-Stein vector-valued inequalities on the $p$-convexification of ball quasi-Banach function spaces. The first is for the fractional maximal operator and the second one is for the Hardy-Littlewood maximal operator.

\begin{proposition} \label{vector-valued fract}
Let $0 < \alpha < n$, $1 < u < \infty$ and $0 < p_0 < q_0 \leq 1$ be such that $\frac{1}{p_0} - \frac{1}{q_0} = \frac{\alpha}{n}$, and let $X$ and $Y$ be ball quasi-Banach function spaces such that $X^{1/p_0}$ and $Y^{1/q_0}$ are ball Banach function spaces and 
$(Y^{1/q_0})' = ((X^{1/p_0})')^{p_0/q_0}$. If the Hardy-Littlewood maximal operator $M$ is bounded on $(Y^{1/q_0})'$, then for any 
$\sigma > \frac{1}{p_0}$ and any sequence of measurable functions $\{ f_j \}_{j=1}^{\infty}$ with $\left\{ \sum_{j=1}^{\infty} |f_j|^{u} \right\}^{1/u} \in X^{\sigma}$,
\begin{equation} \label{feff-stein fract max}
\left\Vert  \left\{ \sum_{j=1}^{\infty} (M_{\frac{\alpha}{\sigma}} f_j)^{u} \right\}^{1/u} \right\Vert_{Y^{\sigma}} \lesssim
\left\Vert \left\{ \sum_{j=1}^{\infty} |f_j|^{u} \right\}^{1/u}  \right\Vert_{X^{\sigma}}.
\end{equation}
\end{proposition}

\begin{proof}
Given $0 < \alpha < n$ and $\sigma > \frac{1}{p_0} = \frac{1}{q_0} + \frac{\alpha}{n}$, we define
\[
\mathcal{F}_{\alpha} = \left\{ \left( \left\{ \sum_{j=1}^{N} (M_{\frac{\alpha}{\sigma}}f_j)^{u} \right\}^{1/u}, 
\left\{ \sum_{j=1}^{N} |f_j|^{u} \right\}^{1/u} \right) : N \in \mathbb{N}, \{f_j \}_{j=1}^{N} \subset L^{\infty}_{comp} \right\},
\]
where $L^{\infty}_{comp}$ denotes the set of all bounded measurable functions on $\mathbb{R}^n$ with compact support. We observe 
that $L^{\infty}_{comp} \subset X^{\sigma}$, and for any $\{f_j \}_{j=1}^{N} \subset L^{\infty}_{comp}$ it is easy to check that 
$\left\{ \sum_{j=1}^{N} |f_j|^{u} \right\}^{1/u} \in X^{\sigma}$.

It is clear that $1 < \sigma p_0 < \frac{\sigma n}{\alpha}$, now for any $v \in \mathcal{A}_1$ one has that 
$v^{1/\sigma q_0} \in \mathcal{A}_{\sigma p_0, \sigma q_0}$ (for the definitions of the classes $\mathcal{A}_{1}$ and $\mathcal{A}_{p,q}$ see \cite[Chapter 7]{grafakos} and \cite[inequality (1.1)]{Muck}, respectively), it follows from \cite[Theorem 3]{Muck} and \cite[Theorem 3.23]{Cruz} that there exists a universal constant $C > 0$ such that 
for any $(F, G) \in \mathcal{F}_{\alpha}$ and any $v \in \mathcal{A}_1$ 
\begin{equation} \label{weighted fract ineq}
\int [F(x)]^{\sigma q_0} v(x) \, dx \leq C \left( \int [G(x)]^{\sigma p_0} [v(x)]^{p_0/q_0} \, dx \right)^{q_0/p_0}.
\end{equation}
On the other hand, by Lemma \ref{doble dual}, we have
\begin{equation} \label{norma}
\| F \|^{\sigma q_0}_{Y^{\sigma}} =  \| F^{\sigma q_0} \|_{Y^{1/q_0}}
\leq C \sup \left\{ \int_{\mathbb{R}^{n}} \left| [F(x)]^{\sigma q_0} g(x) \right| dx : \| g \|_{(Y^{1/q_0})'} \leq 1 \right\}
\end{equation}
for some constant $C > 0$.

Let $\mathcal{R}$ be the operator defined on $(Y^{1/q_0})'$ by
\[
\mathcal{R}g(x) = \sum_{k=0}^{\infty} \frac{M^{k}g(x)}{2^{k} \| M \|_{(Y^{1/q_0})'}^{k}},
\]
where, for $k \geq 1$, $M^{k}$ denotes $k$ iterations of the Hardy-Littlewood maximal operator $M$, $M^{0} = Id$, and 
$\| M \|_{(Y^{1/q_0})'}$ is the operator norm of the maximal operator $M$ on $(Y^{1/q_0})'$. It follows immediately from this 
definition that:

$(i)$ if $g$ is non-negative, $g(x) \leq \mathcal{R}g(x)$ a.e. $x \in \mathbb{R}^{n}$;

$(ii)$ $\| \mathcal{R}g \|_{(Y^{1/q_0})'} \leq 2 \| g \|_{(Y^{1/q_0})'}$; 

$(iii)$ $\mathcal{R}g \in \mathcal{A}_1$ with $[\mathcal{R}g]_{\mathcal{A}_1} \leq 2 \| M \|_{(Y^{1/q_0})'}$.
\\
Since $F$ is non-negative, we can take the supremum in (\ref{norma}) over those non-negative $g$ only. For any fixed non-negative 
$g \in (Y^{1/q_0})'$, by $(i)$ above we have that
\begin{equation} \label{int g}
\int [F(x)]^{\sigma q_0} g(x) dx \leq \int [F(x)]^{\sigma q_0} (\mathcal{R}g)(x) dx.
\end{equation}
Combining $(iii)$, (\ref{weighted fract ineq}), and H\"older's inequality (\cite[Lemma 2.5]{Zhang}), we obtain
\begin{equation} \label{int Rg}
\int [F(x)]^{\sigma q_0} (\mathcal{R}g)(x) dx \leq C \left( \int [G(x)]^{\sigma p_0} [(\mathcal{R}g)(x)]^{p_0 / q_0} dx \right)^{q_0/p_0} 
\end{equation}
\[
\leq C \| G^{\sigma p_0} \|_{X^{1/p_0}}^{q_0/p_0} 
\|(\mathcal{R}g)^{p_0/q_0} \|_{(X^{1/p_0})'}^{q_0/p_0}
\]
\[
= C \| G \|^{\sigma q_0}_{X^{\sigma}} 
\|\mathcal{R}g \|_{((X^{1/p_0})')^{p_0/q_0}}.
\]
By hypothesis we have $((X^{1/p_0})')^{p_0/q_0} = (Y^{1/q_0})'$, so (\ref{int Rg})
\[
\leq C \| G \|^{\sigma q_0}_{X^{\sigma}} \| \mathcal{R}g \|_{(Y^{1/q_0})'}
\]
now, $(ii)$ gives
\[
\leq C \| G \|^{\sigma q_0}_{X^{\sigma}} \| g \|_{(Y^{1/q_0})'}.
\]
Thus, (\ref{int g}) and (\ref{int Rg}) lead to
\begin{equation} \label{norma2}
\int [F(x)]^{\sigma q_0} g(x) dx \leq C \| G \|^{\sigma q_0}_{X^{\sigma}},
\end{equation}
for all non-negative $g$ such that $\| g \|_{(Y^{1/q_0})'} \leq 1$. Then, (\ref{norma}) and (\ref{norma2}) give (\ref{feff-stein fract max}) 
for all finite sequences $\{f_j \}_{j=1}^{N} \subset L^{\infty}_{comp} \cap X^{\sigma}$. By passing to the limit, we 
obtain (\ref{feff-stein fract max}) for all infinite sequences $\{f_j \}_{j=1}^{\infty} \subset L^{\infty}_{comp}$ with 
$\left\{ \sum_{j=1}^{\infty} |f_j|^{u} \right\}^{1/u} \in X^{\sigma}$. For the general case, consider 
$f_{j, N} = f_j \chi_{\{x : |x|<N , |f_j(x)|<N \}}$, since $\left\{ \sum_{j=1}^{\infty} (M_{\frac{\alpha}{\sigma}} f_{j, N})^{u} \right\}^{1/u} \uparrow \left\{ \sum_{j=1}^{\infty} (M_{\frac{\alpha}{\sigma}} f_j)^{u} \right\}^{1/u}$ and $\left\{ \sum_{j=1}^{\infty} |f_{j, N}|^{u} \right\}^{1/u} \uparrow \left\{ \sum_{j=1}^{\infty} |f_j|^{u} \right\}^{1/u}$ as $N \to \infty$, then (\ref{feff-stein fract max}) holds for any sequence of measurable functions $\{ f_j \}_{j=1}^{\infty}$ with $\left\{ \sum_{j=1}^{\infty} |f_j|^{u} \right\}^{1/u} \in X^{\sigma}$.
\end{proof}

Proceeding as in the proof of Proposition \ref{vector-valued fract}, but considering now the $\mathcal{A}_p$ class instead of the 
$\mathcal{A}_{p, q}$ class and \cite[Theorem 9]{Muckenh} and \cite[Corollary 3.12]{Cruz} instead of \cite[Theorem 3]{Muck} and 
\cite[Theorem 3.23]{Cruz} respectively, we obtain the following result.

\begin{proposition} \label{vector-valued H-L max}
Let $1 < u < \infty$, $0 < p_0 \leq 1$ and let $X$ be a ball quasi-Banach function space such that $X^{1/p_0}$ is a ball Banach function space. If the Hardy-Littlewood maximal operator $M$ is bounded on $(X^{1/p_0})'$, then for any 
$\sigma > \frac{1}{p_0}$ and any sequence of measurable functions $\{ f_j \}_{j=1}^{\infty}$ with $\left\{ \sum_{j=1}^{\infty} |f_j|^{u} \right\}^{1/u} \in X^{\sigma}$,
\begin{equation} \label{feff-stein ineq}
\left\Vert  \left\{ \sum_{j=1}^{\infty} (M f_j)^{u} \right\}^{1/u} \right\Vert_{X^{\sigma}} \lesssim
\left\Vert \left\{ \sum_{j=1}^{\infty} |f_j|^{u} \right\}^{1/u}  \right\Vert_{X^{\sigma}}.
\end{equation}
\end{proposition}

\begin{remark} 
Other versions of Proposition \ref{vector-valued H-L max} can be found in \cite[Section 2.5, p. 18]{sawa}.
\end{remark}

\begin{corollary} \label{Acot HL}
Let $0 < p_0 \leq 1$ and let $X$ be a ball quasi-Banach function space such that $X^{1/p_0}$ is a ball Banach function space. If the 
Hardy-Littlewood maximal operator $M$ is bounded on $(X^{1/p_0})'$, then $M$ is bounded on $X^\sigma$ for any $\sigma > \frac{1}{p_0}$.
\end{corollary}

\begin{proof}
Apply Proposition \ref{vector-valued H-L max} with $f_1 \in X^\sigma \setminus \{ 0 \}$ and $f_j = 0$ for all $j \geq 2$.
\end{proof}

\section{Main results} \label{resultados principales}

\subsection{Kernels of type $(\alpha, N)$}

We recall the definition of a kernel of type $(\alpha, N)$ on $\mathbb{R}^n$. Suppose $0 \leq \alpha < n$ and $N \in \mathbb{N}$. 
For $0 < \alpha < n$, a kernel of type $(\alpha, N)$ is a function $K_{\alpha}$ of class $C^{N}$ on $\mathbb{R}^{n} \setminus \{ 0 \}$, which satisfies
\begin{equation} \label{decay}
\left|(\partial^{\beta}K_{\alpha})(x) \right| \lesssim |x|^{\alpha - n - |\beta|} \,\,\,\, \text{for all} \,\, |\beta| \leq N \,\,\, \text{and all} \,\, x \neq 0,
\end{equation}
A kernel of type $(0,N)$ is a distribution $K_{0}$ on $\mathbb{R}^{n}$ which is of class $C^{N}$ on $\mathbb{R}^{n} \setminus \{ 0 \}$, satisfies (\ref{decay}) with $\alpha = 0$, and
\begin{equation} \label{cota L2}
\| K_0 \ast f \|_{L^{2}(\mathbb{R}^{n})} \leq C \| f \|_{L^{2}(\mathbb{R}^{n})}, \,\,\, \text{for all} \,\, f \in \mathcal{S}(\mathbb{R}^{n}).
\end{equation}

\begin{remark} \label{operator Ta}
If $0 < \alpha < n$ and $K_{\alpha}$ is a kernel of type $(\alpha, N)$, from \cite[Proposition 6.2]{Folland}, it follows that the 
operator $T_{\alpha} : f \to K_{\alpha} \ast f$ is a bounded operator $L^{p}(\mathbb{R}^{n}) \to L^{q}(\mathbb{R}^{n})$ for 
$1 < p < \frac{n}{\alpha}$ and $\frac{1}{q} = \frac{1}{p} - \frac{\alpha}{n}$.
\end{remark}

\begin{remark} \label{operator T0}
If $K_0$ is a kernel of type $(0,N)$, by (\ref{cota L2}) it follows that the operator $U_0 : f \to K_0 \ast f$, 
$f \in \mathcal{S}(\mathbb{R}^{n})$, can be extended to a bounded operator on $L^{2}(\mathbb{R}^{n})$, such an extension is unique. We 
denote this extension by $T_0$. Now, it is easy to check that if $a(\cdot) \in L^{2}(\mathbb{R}^{n})$ and their support is contained 
in the cube $Q(x_0, r)$, then $(T_{0} \, a)(x) = (K_{0} \ast a)(x)$ a.e. $x \notin Q(x_0, 2r)$. Moreover, from 
\cite[Theorem 6.10]{Folland}, it follows that $T_0$ is bounded on $L^p(\mathbb{R}^n)$ for every $1 < p < \infty$.
\end{remark}

\subsection{Estimates for the operator $T_{\alpha}$}

Given a kernel $K_{\alpha}$ of type $(\alpha, N)$ with $0 \leq \alpha < n$, we consider the operator $T_{\alpha}$ defined by
\begin{equation} \label{operator T alpha}
T_{\alpha} = \left\{ \begin{array}{c}
                      \text{extension of the operator} \,\, U_0 \,\, \text{on} \,\, L^{2}(\mathbb{R}^{n}), \,\, \text{if} \,\, \alpha = 0 \\
                      \text{convolution operator by} \, K_{\alpha}, \,\, \text{if} \,\, 0 < \alpha < n
											\end{array} \right..
\end{equation}

Our main result is contained in the following theorem.

\begin{theorem} \label{main thm}
Given $0 \leq \alpha < n$ and $N \in \mathbb{N}$, let $T_{\alpha}$ be the convolution operator associated with a kernel $K_{\alpha}$ of type 
$(\alpha, N)$ defined by (\ref{operator T alpha}) and let $X$ be a ball quasi-Banach function space such that $X$ has an absolutely 
continuous quasi-norm, satisfies (\ref{A1}) and (\ref{A2}), and is strictly $s$-convex, where $s \in (0,1]$ is as in (\ref{A1}).

$(i)$ If for some $0 < p_0 \leq 1$, $X^{1/p_0}$ is a ball Banach function space, the Hardy-Littlewood operator $M$ is bounded on 
$(X^{1/p_0})'$ and $N > \max \{d_X + 1, \, n(1/p_0 - 1)\}$, then the operator $T_{0}$ extends to a bounded operator 
$H_{X}(\mathbb{R}^n) \to X$ and $H_{X}(\mathbb{R}^n) \to H_{X}(\mathbb{R}^n)$.

$(ii)$ Let $Y$ be a ball quasi-Banach function space satisfying (\ref{A2}) with $q > \frac{n}{n-\alpha}$ and $0 < \alpha < n$. If for some 
$0 < p_0 < q_0 \leq 1$ such that $\frac{1}{p_0} - \frac{1}{q_0} = \frac{\alpha}{n}$, $X^{1/p_0}$ and $Y^{1/q_0}$ are ball Banach function spaces such that $(Y^{1/q_0})' = ((X^{1/p_0})')^{p_0/q_0}$, the Hardy-Littlewood maximal operator $M$ is bounded on $(Y^{1/q_0})'$ and 
$N > \max \{d_X + 1, \, n(1/p_0 - 1)\}$, then the operator $T_{\alpha}$ extends to a bounded operator $H_{X}(\mathbb{R}^n) \to Y$ and 
$H_{X}(\mathbb{R}^n) \to H_{Y}(\mathbb{R}^n)$.
\end{theorem}

\begin{proof}
We will first prove, using the atomic decomposition of $H_{X}(\mathbb{R}^n)$, that the operator $T_{\alpha}$, $0 < \alpha < n$, extends to a bounded operator $H_{X}(\mathbb{R}^{n}) \to Y$ and that $T_0$, $\alpha = 0$, extends to a bounded operator $H_{X}(\mathbb{R}^{n}) \to X$.
To this end, given $0 \leq \alpha < n$, let $K_{\alpha}$ be a kernel of type $(\alpha, N)$ with $N > \max \{d_X + 1, \, n(1/p_0 - 1)\}$, where $\frac{1}{p_0} = \frac{1}{q_0} + \frac{\alpha}{n}$ is as in $(i)$ or $(ii)$ depending on the case. Then, given 
$f \in H_{X}(\mathbb{R}^{n}) \cap L^{p} (\mathbb{R}^{n})$ (with $p > 1$), by Theorem \ref{X atomic decomp} with $d=N - 1$, there exist sequences of nonnegative numbers $\{ \lambda_j \}_{j=1}^{\infty}$, cubes $\{ Q_j \}_{j=1}^{\infty}$ and $(X, \infty, N-1)$-atoms 
$a_j$ supported on $Q_j$ such that $f = \displaystyle{\sum_{j=1}^{\infty} \lambda_j a_j}$ converges in $L^{p}(\mathbb{R}^{n})$ and
\begin{equation} \label{atomic HX norm}
\left\| \left\{ \sum_{j=1}^{\infty} \left( \frac{\lambda_j}{\| \chi_{Q_j} \|_X} \right)^s \chi_{Q_j} \right\}^{1/s} \right\|_{X} 
\lesssim_{s} \| f \|_{H_{X}(\mathbb{R}^{n})}.
\end{equation}
For $0 < \alpha < n$, we have that $Y$ satisfies (\ref{A2}), with $q > \frac{n}{n - \alpha}$. For such a $q$, 
we put $\frac{1}{p} := \frac{1}{q} + \frac{\alpha}{n}$. Then, by Remark \ref{operator Ta}, we have that the operator $T_{\alpha}$ is 
bounded from $L^{p}(\mathbb{R}^{n})$ to $L^{q}(\mathbb{R}^{n})$. In the case $\alpha = 0$, by Remark \ref{operator T0}, the operator 
$T_0$ is bounded on $L^{p}(\mathbb{R}^{n})$ for every $1 < p < \infty$. In this case, we consider $p = q$ where $q$ is as in (\ref{A2}). 
Since $f = \displaystyle{\sum_{j=1}^{\infty} \lambda_j a_j}$ converges in $L^{p}(\mathbb{R}^{n})$, according to the case $0 < \alpha < n$ 
or $\alpha = 0$, we have that $T_{\alpha}f = \sum_{j} \lambda_j T_{\alpha}a_j$ converges in $L^{q}(\mathbb{R}^{n})$, and so
\[
|T_{\alpha} f(x)| \leq \sum_j \lambda_j |T_{\alpha} a_j(x)|, \,\,\, a.e. \,\, x \in \mathbb{R}^{n}.
\]
Then, for $0 < \alpha < n$,
\begin{equation} \label{I1 I2}
\| T_{\alpha} f \|_{Y} \lesssim \left\| \sum_{j} \lambda_j \chi_{2Q_{j}} |T_{\alpha} a_j| \right\|_{Y}  + 
\left\| \sum_j \lambda_j \chi_{\mathbb{R}^{n} \setminus 2Q_{j}} |T_{\alpha} a_j| \right\|_{Y} =: I_1 + I_2,
\end{equation}
and, for $\alpha = 0$, we have
\begin{equation} \label{I1 I2 tilde}
\| T_{0} f \|_{X} \lesssim \left\| \sum_{j} \lambda_j \chi_{2Q_{j}} |T_{0} a_j| \right\|_{X}  + 
\left\| \sum_j \lambda_j \chi_{\mathbb{R}^{n} \setminus 2Q_{j}} |T_{0} a_j| \right\|_{X} =: \widetilde{I}_1 + \widetilde{I}_2,
\end{equation}
where $2Q_j = Q(z_j, 2r_j)$. To estimate $I_1$, we first apply Remark \ref{operator Ta} to the expression 
$\chi_{2 Q_j} \cdot T_{\alpha} a_j $ followed by Remark \ref{infinite atom} to obtain
\[
\left\| T_{\alpha} a_j \right\|_{L^{q}(2Q_{j})} \lesssim \left\| a_j \right\|_{L^{p}} \lesssim 
\frac{| Q_j |^{\frac{1}{p}}}{\left\| \chi _{Q_j }\right\|_{X}} \lesssim 
\frac{\left| 2Q_{j} \right|^{\frac{1}{q}+\frac{\alpha}{n}}}{\left\| \chi_{Q_{j}} \right\|_{X}},
\]
so
\begin{equation} \label{average}
\left( \frac{1}{|2Q_{j}|} \int_{2Q_j} |T_{\alpha} a_j|^{q}  \right)^{1/q} \lesssim 
\frac{\left| Q_{j} \right|^{\frac{\alpha}{n}}}{\left\| \chi_{Q_{j}} \right\|_{X}}.
\end{equation}
Applying Lemma \ref{q estimate} with $Y$ in place of $X$, followed by (\ref{average}), Lemma \ref{Qalfan}, 
the $s$-inequality for $s \in (0, 1]$, and (\ref{atomic HX norm}), we have
\begin{equation} \label{I1}
I_1 = \left\| \sum_{j} \lambda_j \, \chi_{2 Q_j} \cdot T_{\alpha} a_j  \right\|_{Y} \lesssim 
\left\| \sum_{j} \lambda_j \left( \frac{1}{|2Q_{j}|} \int_{2Q_j} |T_{\alpha} a_j|^{q}  \right)^{1/q} \chi_{2 Q_j} 
\right\|_{Y} 
\end{equation}
\[
\lesssim \left\| \sum_{j} \lambda_j \frac{\left| Q_{j} \right|^{\frac{\alpha}{n}}}{\left\| \chi_{Q_{j}} \right\|_{X}} 
\chi_{2 Q_j} \right\|_{Y} \lesssim
\left\| \sum_{j}  \frac{\lambda_j}{\left\| \chi_{Q_{j}} \right\|_{X}} \chi_{ Q_j} \right\|_{X}
\]
\[
\lesssim \left\| \left\{ \sum_{j} \left( \frac{\lambda_j }{\| \chi_{Q_j} \|_{X}} \right)^s \chi_{Q_j} \right\}^{1/s} 
\right\|_{X} \lesssim_{s} \| f \|_{H_{X}(\mathbb{R}^n)}.
\]
Since (\ref{average}) holds with $\alpha = 0$ and $p = q$, proceeding as in the estimate of (\ref{I1}), we obtain
\begin{equation} \label{I1tilde}
\widetilde{I_1} = \left\| \sum_{j} \lambda_j \chi_{2Q_{j}} |T_{0} a_j| \right\|_{X} \lesssim_{s} \| f \|_{H_{X}(\mathbb{R}^n)}.
\end{equation}
Now, we estimate the terms $I_2$ and $\widetilde{I}_2$. Since each $a_j$ is an $(X, \infty, N-1)$-atom with 
$N > \max \{d_X + 1, n(1/p_0 - 1)\}$ and support on the cube $Q_j = Q(z_j, r_j)$, using the moment conditions satisfied by $a_j$ 
(i.e. $a_j \perp \mathcal{P}_{N-1}$), and Remark \ref{operator Ta} (for $0 < \alpha < n$) or Remark \ref{operator T0} (for $\alpha = 0$), we obtain
\[
T_{\alpha}a_j(x) = \int_{Q_j} \left( K_{\alpha}(x-y) - q_{x, N}(y) \right) a_j(y) dy, \,\,\,\,\,\, \textit{for all} \,\,\, x \notin 2Q_j,
\]
where $q_{x, N}$ is the Taylor polynomial of degree $N-1$ of the function $y \rightarrow K_{\alpha}(x-y)$ expanded around $z_j$. By the standard estimate of the remainder term of the Taylor expansion and (\ref{decay}), for any $y  \in Q_j$ and any $x \notin 2Q_j$, we get
\[
\left|K_{\alpha}(x-y) - q_{x, N}(y) \right| \leq C r^{N}_j |x - z_j|^{-n + \alpha - N}.
\]
This inequality, together with $\| a_j \|_{L^{\infty}} \leq \| \chi_{Q_j} \|_{X}^{-1}$, allows us to conclude that
\begin{equation} \label{Talfa puntual}
|T_{\alpha} a_j(x) | \lesssim \frac{r^{n+N}_j}{\| \chi_{Q_j} \|_{X}} |x-z_j|^{-n + \alpha - N}
\lesssim \frac{\left[ M_{\frac{\alpha n}{n+N}}(\chi_{Q_j}) (x) \right]^{\frac{n+N}{n}}}{\| \chi_{Q_j} \|_{X}},
\end{equation}
for all $x \notin 2Q_j$. Putting $\sigma= \frac{n+N}{n}$, (\ref{Talfa puntual}) leads to
\[
I_2 \lesssim \left\| \left\{ \sum_j  \frac{\lambda_j}{\| \chi_{Q_j} \|_{X}} \left[ M_{\frac{\alpha}{\sigma}}(\chi_{Q_j}) \right]^{\sigma} \right\}^{1/\sigma} \right\|^{\sigma}_{Y^{\sigma}}. 
\]
Since $\sigma = \frac{n+N}{n} > \frac{1}{p_0}$, applying Proposition \ref{vector-valued fract} with $u = \sigma$, we obtain
\begin{equation} \label{I2}
I_2 \lesssim \left\| \left\{ \sum_j  \frac{\lambda_j}{\| \chi_{Q_j} \|_{X}} \chi_{Q_j} \right\}^{1/\sigma} 
\right\|^{\sigma}_{X^{\sigma}} = \left\| \sum_j \frac{\lambda_j}{\| \chi_{Q_j} \|_{X}} \chi_{Q_j} \right\|_{X}
\lesssim_{s} \| f \|_{H_{X}(\mathbb{R}^n)}. 
\end{equation}
Since (\ref{Talfa puntual}) also holds for $\alpha = 0$, we proceed as in the estimate of $I_2$, but now applying 
Lemma \ref{vector-valued H-L max}, we obtain
\begin{equation} \label{I2tilde}
\widetilde{I}_2 = \left\| \sum_j \lambda_j \chi_{\mathbb{R}^{n} \setminus 2Q_{j}} |T_{0} a_j| \right\|_{X} 
\lesssim_{s} \| f \|_{H_{X}(\mathbb{R}^n)}.
\end{equation}
So, by (\ref{I1 I2}), (\ref{I1}), (\ref{I2}) and Proposition \ref{dense set}, the operator $T_{\alpha}$, $0 < \alpha < n$, extends to a bounded operator $H_{X}(\mathbb{R}^n) \to Y$. Now, by (\ref{I1 I2 tilde}), (\ref{I1tilde}), (\ref{I2tilde}) and Proposition \ref{dense set}, we have that $T_0$ extends to a bounded operator $H_{X}(\mathbb{R}^n) \to X$.

\

Now, by using the atomic decomposition of $H_{X}(\mathbb{R}^n)$ and the maximal characterization of $H_{Y}(\mathbb{R}^n)$ 
and $H_{X}(\mathbb{R}^n)$, we will prove that the operator $T_{\alpha}$, $0 < \alpha < n$, extends to a bounded operator 
$H_{X}(\mathbb{R}^{n}) \to H_{Y}(\mathbb{R}^{n})$ and that $T_0$ extends to a bounded operator 
$H_{X}(\mathbb{R}^{n}) \to H_{X}(\mathbb{R}^{n})$. Let $0 \leq \alpha < n$ and let $\frac{1}{p_0} = \frac{1}{q_0} + \frac{\alpha}{n}$ be 
as in $(i)$ or $(ii)$ according to the case. Given $f \in H_{X}(\mathbb{R}^{n}) \cap L^{p} (\mathbb{R}^{n})$, we consider its atomic decomposition $f = \sum_j \lambda_j a_j$, and $N > \max \{d_X + 1, \, n(1/p_0 - 1)\}$. Now, by Corollary \ref{Acot HL}, we have that, for 
$r = \frac{p_0}{p_0 + 1}$ ($\frac{1}{p_0} = \frac{1}{q_0} + \frac{\alpha}{n}$), the Hardy-Littlewood maximal operator $M$ is bounded 
on $X^{1/r}$ or on $Y^{1/r}$. Then, by Remark \ref{Hx space}, we have that, for 
$L \geq \left\lfloor \frac{n(p_0 + 1)}{p_0} + 3 \right\rfloor$, 
$\| f \|_{H_X (\mathbb{R}^{n})} = \| \mathcal{M}^{0}_{L} f \|_X$ for $f \in H_X (\mathbb{R}^{n})$ and, respectively, for 
$g \in H_Y (\mathbb{R}^{n})$.

On the other hand, for $0 \leq \alpha < n$, we put $\frac{1}{p} := \frac{1}{q} + \frac{\alpha}{n}$, where for $0 < \alpha < n$, $q$ is as in 
(\ref{A2}) but with $Y$ instead of $X$ and $q > \frac{n}{n - \alpha}$; for $\alpha =0$, $q$ is as in (\ref{A2}). Since $T_{\alpha}$ is bounded from $L^{p}(\mathbb{R}^{n})$ into $L^{q}(\mathbb{R}^{n})$, $H^{q}(\mathbb{R}^{n}) \equiv L^{q}(\mathbb{R}^{n})$ with equivalent norms, and $T_{\alpha}f = \sum_{j} \lambda_j T_{\alpha}a_j$ converges in $L^{q}(\mathbb{R}^{n})$, it follows, for $0 \leq \alpha < n$, that
\[
\mathcal{M}^{0}_{L}(T_{\alpha} f)(x) \leq \sum_{j=1}^{\infty} \lambda_j \mathcal{M}^{0}_{L}(T_{\alpha} a_j)(x), \,\,\, a.e. \,\, x \in \mathbb{R}^{n}.
\]
Then, for $0 < \alpha < n$,
\[
\| T_{\alpha} f \|_{H_{Y}} = \| \mathcal{M}^{0}_{L}(T_{\alpha} f) \|_{Y} \lesssim \left\| \sum_{j} \lambda_j \chi_{2 Q_{j}} 
\mathcal{M}^{0}_L(T_{\alpha} a_j) \right\|_{Y}  
\]
\[
+ 
\left\| \sum_j \lambda_j \chi_{\mathbb{R}^{n} \setminus 2 Q_{j}} \mathcal{M}^{0}_L(T_{\alpha} a_j) \right\|_{Y} =: J_1 + J_2.
\]
For $\alpha = 0$ and $L$ as above, we have
\[
\| T_{0} f \|_{H_{X}} = \| \mathcal{M}^{0}_{L}(T_{0} f) \|_{X} \lesssim \left\| \sum_{j} \lambda_j \chi_{2 Q_{j}} 
\mathcal{M}^{0}_L(T_{0} a_j) \right\|_{X}  
\]
\[
+ 
\left\| \sum_j \lambda_j \chi_{\mathbb{R}^{n} \setminus 2 Q_{j}} \mathcal{M}^{0}_L(T_{0} a_j) \right\|_{X} =: \widetilde{J}_1 + \widetilde{J}_2.
\]
To estimate $J_1$ or $\widetilde{J}_1$, we observe that for $0 \leq \alpha < n$ and $\frac{1}{q} = \frac{1}{p} - \frac{\alpha}{n}$, where 
$q > \frac{n}{n - \alpha}$ is the parameter appearing in (\ref{A2}) for $X$ or $Y$, depending on the case,
\[
\left\| \mathcal{M}^{0}_{L} (T_{\alpha} a_j) \right\|_{L^{q}(2Q_{j})} \lesssim 
\left\| T_{\alpha} a_j \right\|_{L^{q}} \lesssim \left\| a_j \right\|_{L^{p}}
\lesssim 
\frac{| Q_j |^{\frac{1}{p}}}{\left\| \chi _{Q_j }\right\|_{X}} \lesssim 
\frac{\left| 2Q_{j} \right|^{\frac{1}{q}+\frac{\alpha}{n}}}{\left\| \chi_{Q_{j}} \right\|_{X}}.
\]
Then, by proceeding as above in the estimate of $I_1$ or $\widetilde{I}_1$, according to the case, we get
\[
J_1, \,\, \widetilde{J}_1 \lesssim \left\| \left\{ \sum_{j} \left( \frac{\lambda_j }{\| \chi_{Q_j} \|_{X}} \right)^s 
\chi_{Q_j} \right\}^{1/s} \right\|_{X} \lesssim_{s} \| f \|_{H_{X}}.
\]
We now estimate $J_2$ and $\widetilde{J}_2$. For $x \notin 2 Q_j = 2 Q_j(z_j, r_j)$ and every $t > 0$, by the $(N-1)$-moment condition of the atoms, 
we have
\[
((T_{\alpha} a_j) \ast \phi_t)(x) = \int_{Q_j} a_j(y) (K_{\alpha} \ast \phi_t)(x-y) \, dy  
\]
\[
= \int_{Q_j} a_j(y) \left[(K_{\alpha} \ast \phi_t)(x-y) - \widetilde{q}_{x, t, N}(y) \right] \, dy, 
\] 
where $y \to \widetilde{q}_{x, t, N}(y)$ is the Taylor polynomial of degree $N-1$ of the function $y \to (K_{\alpha} \ast \phi_t)(x-y)$, centered at $z_j$. Let $\phi \in \mathcal{S}(\mathbb{R}^{n})$ with $\| \phi \|_{\mathcal{S}(\mathbb{R}^{n}), \, L} \leq 1$ 
(where $L \geq \left\lfloor \frac{n(p_0 + 1)}{p_0} + 3 \right\rfloor$ and sufficiently large). Here $L$ plays the role of $N$ in 
\cite[Lemma 6.4]{Folland}. By \cite[Lemma 6.9]{Folland} applied with $G=\mathbb{R}^n$, $r = N$ and $0 \leq \alpha < n$, we have
\[
|\partial^{\beta}(K_{\alpha} \ast \phi_t)(u)| = |((\partial^{\beta} K_{\alpha}) \ast \phi_t)(u)|
\lesssim \| \phi \|_{\mathcal{S}(\mathbb{R}^{n}), \, L} |u|^{\alpha - n - |\beta|} \lesssim |u|^{\alpha - n - |\beta|},
\]
for all $u \neq 0$, $t > 0$ and $|\beta| \leq N$. Then, by the standard estimate of the remainder term of the Taylor expansion, for any 
$y  \in Q_j$ and any $x \notin 2Q_j$, we obtain
\[
\left| (K_{\alpha} \ast \phi_t)(x-y) - \widetilde{q}_{x, t, N}(y) \right| \lesssim r^{N}_j  |x-z_j|^{-n+ \alpha - N},
\]
and so, for $0 \leq \alpha < n$ and $x \notin 2 Q_j$,
\[
|((T_{\alpha} a_j) \ast \phi_t)(x)| \lesssim \frac{\left[ M_{\frac{\alpha n}{n+N}}(\chi_{Q_j}) (x) \right]^{\frac{n+N}{n}}}{\| \chi_{Q_j} \|_{X}}.
\]
This estimate does not depend on $t>0$ and $\phi$ with $\| \phi \|_{\mathcal{S}(\mathbb{R}^{n}), \, L} \leq 1$. Then, adapting the arguments used above to estimate $I_2$ and $\widetilde{I}_2$, and taking the supremum on $t>0$ and $\phi \in \mathcal{F}_L$, we obtain
\[
J_2, \, \widetilde{J}_2 \lesssim \left\| \left\{ \sum_{j} \left( \frac{\lambda_j }{\| \chi_{Q_j} \|_{X}} \right)^s 
\chi_{Q_j} \right\}^{1/s} \right\|_{X} \lesssim_{s} \| f \|_{H_{X}}.
\]
Finally, by the estimates of $J_1, J_2, \widetilde{J}_1$, $\widetilde{J}_2$ and Proposition \ref{dense set}, we have that $T_{\alpha}$, $0 < \alpha < n$, extends to a bounded operator $H_{X}(\mathbb{R}^{n}) \to H_{Y}(\mathbb{R}^{n})$ and $T_0$ extends to a bounded operator $H_{X}(\mathbb{R}^{n}) \to H_{X}(\mathbb{R}^{n})$. This concludes the proof.
\end{proof}

\begin{remark} \label{note thm}
If $X$ is a ball quasi-Banach function space such that, for some $r > 1$, $X^{1/r}$ is a ball Banach function space and the Hardy-Littlewood maximal operator $M$ is bounded on $X^{1/r}$, then by \cite[Theorem 3.4]{sawa} we have that $H_X$ is isomorphic to $X$ with equivalent norms. Consequently, if, in addition to the assumptions of Theorem \ref{main thm}(i), the maximal operator $M$ is bounded on $X^{1/r}$ for some 
$r > 1$, then it reduces to the $X \to X$ boundedness of the operator $T_0$. Similarly, if for some $\frac{n}{n+\alpha} < s \leq 1$ the maximal operator $M$ is bounded on $Y^{1/r}$ with $\frac{1}{r} = \frac{1}{s} - \frac{\alpha}{n}$ or for some $1 < s < \frac{n}{\alpha}$, 
$M$ is bounded on $X^{1/s}$ and on $Y^{1/r}$ with $\frac{1}{r} = \frac{1}{s} - \frac{\alpha}{n}$, then Theorem \ref{main thm}(ii) reduces to the $H_X \to Y$ or the $X \to Y$ boundedness of the operator $T_{\alpha}$ respectively.
\end{remark}

The following example illustrates Theorem \ref{main thm} and Remark \ref{note thm}. More detailed examples will be given in Section 
\ref{4 examples}.

\begin{example} Let $X = L^{p}(\mathbb{R}^n)$, where $0 < p < \infty$. Then $H_X = H^p(\mathbb{R}^n)$ is the classical Hardy space. It is well known that $H^{p}(\mathbb{R}^{n}) \cong L^{p}(\mathbb{R}^{n})$ for $1 < p < \infty$, $H^{1}(\mathbb{R}^{n}) \subset L^{1}(\mathbb{R}^{n})$ strictly, and the spaces $H^{p}(\mathbb{R}^{n})$ and $L^{p}(\mathbb{R}^{n})$ are not comparable for $0 < p < 1$. Then, for $X = L^{p}(\mathbb{R}^n)$ and $Y = L^q(\mathbb{R}^n)$ with $\frac{1}{q} = \frac{1}{p} - \frac{\alpha}{n}$ and $0 \leq \alpha < n$, one has the following mapping properties for the operator $T_{\alpha}$:
\[
T_0 : H^{p}(\mathbb{R}^n) \to H^{p}(\mathbb{R}^n), \,\,\, \text{if} \,\,\, 0 < p \leq 1 \,\,\, (\text{Theorem \ref{main thm} - (i)}),
\]
\[
T_0 : H^{p}(\mathbb{R}^n) \to L^{p}(\mathbb{R}^n), \,\,\, \text{if} \,\,\, 0 < p \leq 1 \,\,\, (\text{Theorem \ref{main thm} - (i)}),
\]
\[
T_0 : L^{p}(\mathbb{R}^n) \to L^{p}(\mathbb{R}^n), \,\,\, \text{if} \,\,\, 1 < p < \infty \,\,\, (\text{Remark \ref{note thm}}),
\]
\[
T_\alpha : H^{p}(\mathbb{R}^n) \to H^{q}(\mathbb{R}^n), \,\,\, \text{if} \,\,\, 0 < \alpha < n \,\,\, \text{and} \,\,\, 
0 < p \leq \frac{n}{n+\alpha} \,\,\, (\text{Theorem \ref{main thm} - (ii)}),
\]
\[
T_\alpha : H^{p}(\mathbb{R}^n) \to L^{q}(\mathbb{R}^n), \,\,\, \text{if} \,\,\, 0 < \alpha < n \,\,\, \text{and} \,\,\, 
0 < p \leq \frac{n}{n+\alpha} \,\,\, (\text{Theorem \ref{main thm} - (ii)}),
\]
\[
T_\alpha : H^{p}(\mathbb{R}^n) \to L^{q}(\mathbb{R}^n), \,\,\, \text{if} \,\,\, 0 < \alpha < n \,\,\, \text{and} \,\,\, 
\frac{n}{n+\alpha} < p \leq 1 \,\,\, (\text{Remark \ref{note thm}}),
\]
and
\[
T_\alpha : L^{p}(\mathbb{R}^n) \to L^{q}(\mathbb{R}^n), \,\,\, \text{if} \,\,\, 0 < \alpha < n \,\,\, \text{and} \,\,\, 
1 < p < \frac{n}{\alpha} \,\,\, (\text{Remark \ref{note thm}}).
\]
For more details see Section \ref{4 examples} (Weighted Hardy spaces) below.
\end{example}

\subsection{Singular integrals} \label{Sing integ} Let $\Omega \in C^{\infty}(S^{n-1})$ with $\int_{S^{n-1}} \Omega(u) d\sigma(u)=0$. We define the singular integral operator $T$ by
\begin{equation}
Tf(x) = \lim_{\epsilon \rightarrow 0^{+}} \int_{|y| > \epsilon} \frac{\Omega(y/|y|)}{|y|^{n}} f(x-y) \, dy, \,\,\,\, x \in \mathbb{R}^{n}. \label{T}
\end{equation}
It is well known that $\widehat{Tf}(\xi) = m(\xi) \widehat{f}(\xi)$, where the multiplier $m$ is homogeneous of degree zero and belongs to 
$C^{\infty}(\mathbb{R}^{n} \setminus \{0\})$. Moreover, if $k(y) = \frac{\Omega(y/|y|)}{|y|^{n}}$ we have
\begin{equation} \label{k estimate}
|\partial^{\beta} k(y) |\leq C |y|^{-n-|\beta|}, \,\,\,\,  \textit{for all } \,\, y \neq 0 \,\, \textit{and all multi-index} \,\, \beta. 
\end{equation}
Then, the operator $T$ is bounded on $L^{p}(\mathbb{R}^{n})$ for every $1 < p < +\infty$ (see \cite{elias}).

\begin{theorem}  \label{Sing estimates}
Let $X$ be a ball quasi-Banach function space as in Theorem \ref{main thm}. Then 
the singular integral operator $T$ given by (\ref{T}) extends to a bounded operator $H_{X}(\mathbb{R}^n) \to X$ and 
$H_{X}(\mathbb{R}^n) \to H_{X}(\mathbb{R}^n)$.
\end{theorem}

\begin{proof}
We fix $\eta \in C^{\infty}(0, +\infty)$ such that $\eta(r) = 0$ if $0 < r \leq 1/2$ and $\eta(r) = 1$ if $r \geq 1$. Then, for 
$x \in \mathbb{R}^n \setminus \{ 0\}$ and $\epsilon > 0$, we define $h(x) := \eta(|x|^2)$ and $h_{\epsilon}(x) := h(x/\epsilon)$.
Let $k_{\epsilon}(y) := h_{\epsilon}(y) \frac{\Omega(y/|y|)}{|y|^{n}}$ and $T_{\epsilon}f := k_{\epsilon} \ast f$. It is easy to check 
that $k_{\epsilon}$ satisfies (\ref{k estimate}), where the constant $C$ does not depend on $\epsilon$. On the other hand, by 
\cite[Theorem (a)]{Feff}, we have $\|  k_{\epsilon} \ast f\|_2 \leq C \| f \|_2$ for all $f \in \mathcal{S}(\mathbb{R}^n)$, where the constant $C$ is independent of $\epsilon$. Then, applying Theorem \ref{main thm}(i) with $K_{0, \epsilon} = k_{\epsilon}$ and letting $\epsilon \to 0$, the theorem follows.
\end{proof}

\subsection{Fractional integrals} \label{Fract integ} For $0 < \alpha < n$, the Riesz potential $I_{\alpha}$ on $\mathbb{R}^{n}$ is defined by
\begin{equation} \label{Riesz}
(I_{\alpha}f)(x) = \int_{\mathbb{R}^{n}} |x-y|^{\alpha - n} f(y) \, dy.
\end{equation}
It is clear that $|\cdot|^{\alpha - n} \in C^{\infty}(\mathbb{R}^{n} \setminus \{0\}) = \displaystyle{\bigcap_{N \in \mathbb{N}}} 
C^{N}(\mathbb{R}^{n} \setminus \{0\})$ and satisfies the condition (\ref{decay}) for every $N \in \mathbb{N}$. Applying 
Theorem \ref{main thm}(ii) with $K_{\alpha}(\cdot) = |\cdot|^{\alpha - n}$, we obtain the following result.

\begin{theorem}  \label{Riesz estimates}
Let $0 < \alpha < n$, and let $X$ and $Y$ be ball quasi-Banach function spaces as in Theorem \ref{main thm}. Then 
the Riesz potential $I_{\alpha}$ given by (\ref{Riesz}) extends to a bounded operator $H_{X}(\mathbb{R}^n) \to Y$ and 
$H_{X}(\mathbb{R}^n) \to H_{Y}(\mathbb{R}^n)$.
\end{theorem}

\section{Applications} \label{4 examples}

In this section, we illustrate our results with four concrete examples of Hardy-type spaces associated with ball quasi-Banach function spaces 
$X$ and $Y$ satisfying the hypotheses of Theorem \ref{main thm}. More precisely, the following conditions are satisfied:

- $X$ satisfies (\ref{A1}) and (\ref{A2}).

- $X$ is strictly $s$-convex, where $s \in (0,1]$ is as in (\ref{A1}). 

- $\| \cdot \|_X$ is absolutely continuous.

- There exists $0 < p_0 \leq 1$ such that $X^{1/p_0}$ is a ball Banach function space, the Hardy-Littlewood operator $M$ 
is bounded on $(X^{1/p_0})'$.

- For $0 < \alpha < n$, there exist $0 < p_0 < q_0 \leq 1$ such that $\frac{1}{p_0} - \frac{1}{q_0} = \frac{\alpha}{n}$, $Y$ 
satisfies (\ref{A2}) with $q > \frac{n}{n - \alpha}$, $X^{1/p_0}$ and $Y^{1/q_0}$ are ball Banach function spaces such that 
$(Y^{1/q_0})' = ((X^{1/p_0})')^{p_0/q_0}$ and the Hardy-Littlewood operator $M$ is bounded on $(Y^{1/q_0})'$.

\

{\bf Weighted Hardy spaces.} Given $0 < p < \infty$ and a weight $w \in \mathcal{A}_{\infty}$ (see \cite{grafakos}, \cite{Stein}), the weighted Lebesgue space $L^{p}_{w}(\mathbb{R}^n)$ is defined as the set of all measurable functions $f$ such that
\[
\| f \|_{L^{p}_{w}} := \left(  \int_{\mathbb{R}^n} |f(x)|^p w(x) dx \right)^{1/p} < \infty.
\]
By \cite[Section 7.1]{sawa}, the pair $(L^{p}_{w}(\mathbb{R}^n), \| \cdot \|_{L^{p}_{w}})$ is a ball quasi-Banach function space, and its 
quasi-norm $\| \cdot \|_{L^{p}_{w}}$ is \textit{absolutely continuous}. If $p > 1$, then $L^p_w (\mathbb{R}^n)$ is a ball Banach function space with $(L^p_w)' = L^{p'}_{w^{1-p'}}$, where $\frac{1}{p} + \frac{1}{p'} = 1$. In \cite[Section 7.1]{sawa}, it was pointed out that a weighted Lebesgue space may not be a Banach function space. By Fatou's lemma, it follows that for any $s \in (0,p)$, the space 
$L^{p}_{w}(\mathbb{R}^n)$ is \textit{strictly $s$-convex}, where $(L^{p}_{w}(\mathbb{R}^n))^{1/s} = L^{p/s}_{w}(\mathbb{R}^n)$ .

Given $0 \leq \alpha < n$, let $X:= L^{p}_{w}(\mathbb{R}^n)$, with $p \in (0, \frac{n}{\alpha})$ and $w \in \mathcal{A}_{\infty}$. Then, for any $s \in (0, p)$ and $w \in \mathcal{A}_{p/s}$, the Hardy-Littlewood maximal operator $M$ is bounded on $L^{p/s}_w$ and so also on 
$(L^{p/s}_{w})' = L^{(p/s)'}_{w^{1-(p/s)'}}$ since $w^{1-(p/s)'} \in \mathcal{A}_{(p/s)'}$ (see \cite[Theorem 9]{Muckenh}).  
For $0 <  p < \frac{n}{\alpha}$, we define $\frac{1}{q} : = \frac{1}{p} - \frac{\alpha}{n}$ and 
$Y := L^q_{w^{q/p}} (\mathbb{R}^n)$. It is clear that if $\alpha = 0$, then $X=Y$. For any $p_0 \in (0, \min\{ \frac{n}{n+\alpha}, p \})$ fixed, we put $\frac{1}{q_0} := \frac{1}{p_0} - 
\frac{\alpha}{n}$, then $0 < p_0 \leq q_0 \leq 1$, $q_0 \in (0, q)$, $X^{1/p_0}$ and $Y^{1/q_0}$ are ball Banach function spaces and 
$M$ is bounded on $(Y^{1/q_0})'$. Since $\frac{1}{p} - \frac{1}{q} = \frac{\alpha}{n} = \frac{1}{p_0} - \frac{1}{q_0}$, it follows that 
$(Y^{1/q_0})' = ((X^{1/p_0})')^{p_0/q_0}$. For $X = L^p_w (\mathbb{R}^n)$, the assumption (\ref{A1}) holds true for any 
$\theta, s \in (0,1]$, $\theta < s$, $p \in (\theta, \frac{n}{\alpha})$ and $w \in \mathcal{A}_{p/\theta}$ (it follows from 
\cite[Theorem 3.1 (b)]{John} applied on $L^{p/\theta}_w$, with $r = s/\theta$ and $|f_j|^{\theta}$ instead of $f_k$); 
the assumption (\ref{A2}), by duality, \cite[Theorem 9]{Muckenh} and Remark \ref{cond equiv A2}, holds true for any 
$r \in (0, \min\{ 1, p \})$, $w \in \mathcal{A}_{p/r}$, and $\widetilde{q} \in (\max\{1, p\}, \infty)$ sufficiently large such that 
$w^{1-(p/r)'} \in \mathcal{A}_{(p/r)'/(\widetilde{q}/r)'}$.

We now specify weights $w$ for which the spaces $X = L^{p}_w$ and $Y=L^{q}_{w^{q/p}}$ satisfy the hypotheses of 
Theorem \ref{main thm}. Given $0 \leq \alpha < n$, let $p \in (0, \frac{n}{\alpha})$ and $\frac{1}{q} = \frac{1}{p} - \frac{\alpha}{n}$. 
We recall that, for $\alpha = 0$, $X = Y$. Then, for $p_0 \in (0, \frac{1}{2}\min\{ \frac{n}{n+\alpha}, p \})$ and 
$\frac{1}{q_0} = \frac{1}{p_0} - \frac{\alpha}{n}$, we have that $\mathcal{A}_1 \subset \mathcal{A}_{p/p_0} \subset \mathcal{A}_{q/q_0}$. Suppose $w^{q/p} \in \mathcal{A}_{1}$, then $w \in \mathcal{A}_{1}$. Thus, by taking $w^{q/p} \in \mathcal{A}_{1}$, $\theta = p_0$ and 
$s= 2 p_0$, $X= L^{p}_w(\mathbb{R}^n)$ satisfies (\ref{A1}), is strictly $s$-convex, and (\ref{A2}) holds with $r = p_0$ since 
$\widetilde{q}$ can be chosen large enough. Finally, $Y = L^q_{w^{q/p}}(\mathbb{R}^n)$ satisfies (\ref{A2}) with $r = q_0$ and 
$\widetilde{q} > \frac{n}{n - \alpha}$ large enough. Thus, Theorem \ref{main thm} applies to these spaces $X$ and $Y$. We point out that 
$H_{X}(\mathbb{R}^n) = H^{p}_w(\mathbb{R}^n)$ is the weighted Hardy space defined in \cite{tor}. Note that Theorem \ref{Sing estimates}, 
with $X = L^p_{w}$ and $w \in \mathcal{A}_1$, is a particular case of \cite[Theorems 4.1 and 4.4]{Pablo}. Theorem \ref{Riesz estimates}, 
with $X = L^p_{w}$, $Y = L^q_{w^{q/p}}$ and $w^{q/p} \in \mathcal{A}_1$, recover \cite[Theorems (1.4) and (1.5)]{wheeden} for the case 
$u = w^{q/p}$ and $v = w$. Proposition \ref{vector-valued fract} provides an off-diagonal Fefferman-Stein vector-valued inequality for the fractional maximal $M_{\alpha}$ on weighted Lebesgue spaces. This inequality was previously established in \cite[Theorem 3.23]{Cruz}.

\

{\bf Variable Hardy spaces.} Let $p(\cdot) : \mathbb{R}^n \to (0, \infty)$ be a measurable function such that
\[ 
0 < p_{-} := \essinf_{x \in \mathbb{R}^n} p(x) \leq \esssup_{x \in \mathbb{R}^n} p(x) : = p_{+} < \infty.
\]
Then, the variable Lebesgue space $L^{p(\cdot)}(\mathbb{R}^n)$ consists of all the measurable functions $f$ such that
\[
\| f \|_{p(\cdot)} := \inf \left\{ \lambda > 0 : \int_{\mathbb{R}^n} (\lambda^{-1}|f(x)|)^{p(x)} dx \leq 1 \right\} < \infty.
\]
By \cite[Section 7.8]{sawa}, we have that the pair $(L^{p(\cdot)}(\mathbb{R}^n), \| \cdot \|_{p(\cdot)})$ is a ball quasi-Banach function space. Moreover, the quasi-norm $\| \cdot \|_{p(\cdot)}$ is \textit{absolutely continuous}. If $p_{-} > 1$, then $L^{p(\cdot)}(\mathbb{R}^n)$ is a ball Banach function space with $(L^{p(\cdot)}(\mathbb{R}^n))' = L^{p'(\cdot)}(\mathbb{R}^n)$, where 
$\frac{1}{p(x)} + \frac{1}{p'(x)} = 1$ for all $x$. Then, by \cite[Proposition 2.18 and Theorem 2.61]{Fiorenza}, for any $s \in (0, p_{-})$, $(L^{p(\cdot)})^{1/s} = L^{p(\cdot)/s}$ and the space $L^{p(\cdot)}$ is \textit{strictly $s$-convex}.

An exponent $p(\cdot) : \mathbb{R}^n \to (0, \infty)$ is said to be globally log-H\"older continuous if there exist positive constants $C$ and $p_{\infty}$ such that
\[
|p(x) - p(y)| \leq \frac{-C}{\log(|x-y|)}, \,\, \text{for} \,\, |x-y| \leq 1/2
\]
and
\[
|p(x) - p_{\infty}| \leq \frac{C}{\log(e + |x|)}, \,\, \text{for all} \,\, x \in \mathbb{R}^n.
\]
Let $L^{p(\cdot)}(\mathbb{R}^n)$, where $p(\cdot)$ is globally log-H\"older continuous. Then, for 
any $s \in (0, p_{-})$, the Hardy-Littlewood maximal operator $M$ is bounded on $L^{p(\cdot)/s}$ (see \cite[Theorem 1.5]{DCruz}), and 
so also on $(L^{p(\cdot)/s})' = L^{(p(\cdot)/s)'}$, since the exponent $(p(\cdot)/s)'$ results globally log-H\"older continuous with 
$((p(\cdot)/s)')_{-} > 1$. 

Given $0 \leq \alpha < n$, let $p(\cdot)$ be an exponent such that $0 < p_{-} \leq p_{+} < \frac{n}{\alpha}$ and is globally log-H\"older continuous. Then, we define $\frac{1}{q(\cdot)} : = \frac{1}{p(\cdot)} - \frac{\alpha}{n}$, which is globally log-H\"older continuous. 
Let $X:= L^{p(\cdot)}(\mathbb{R}^n)$ and $Y := L^{q(\cdot)}(\mathbb{R}^n)$. It is clear that if $\alpha = 0$, then $X=Y$. For any 
$p_0 \in (0, \min\{ \frac{n}{n+\alpha}, p_{-} \})$ fixed, we put $\frac{1}{q_0} := \frac{1}{p_0} - \frac{\alpha}{n}$, then 
$0 < p_0 \leq q_0 \leq 1$, $q_0 \in (0, q_{-})$, $X^{1/p_0}$ and $Y^{1/q_0}$ are ball Banach function spaces and $M$ is 
bounded on $(Y^{1/q_0})'$. Since $\frac{1}{p(\cdot)} - \frac{1}{q(\cdot)} = \frac{\alpha}{n} = \frac{1}{p_0} - 
\frac{1}{q_0}$, it follows that $(Y^{1/q_0})' = ((X^{1/p_0})')^{p_0/q_0}$. For $X:= L^{p(\cdot)}$, assumption (\ref{A1}) 
holds true for any $s \in (0,1]$ and $\theta \in (0, \min\{s, p_{-}\})$ (indeed, we apply \cite[Lemma 2.4]{Nakai} on $L^{p(\cdot)/\theta}$, with $u=s/\theta$ and $|f_j|^{\theta}$ instead of $f_j$); assumption (\ref{A2}) holds true for any $r \in (0, \min\{ 1, p_{-}\})$ 
and $\widetilde{q} \in (\max \{1, p_{+}\}, \infty)$ (this follows by duality, \cite[Theorem 1.5]{DCruz} and Remark \ref{cond equiv A2}). Similarly, assumption (\ref{A2}) holds true for $Y= L^{q(\cdot)}$. Then, one has that 
$H_{X}(\mathbb{R}^n) = H^{p(\cdot)}(\mathbb{R}^n)$ and $H_{Y}(\mathbb{R}^n) = H^{q(\cdot)}(\mathbb{R}^n)$ are the variable Hardy spaces with exponents $p(\cdot)$ and $q(\cdot)$ defined in \cite{Nakai}.

Then, Theorem \ref{main thm} applies to these spaces $X$ and $Y$. We note that Theorem \ref{Sing estimates}, for $X= L^{p(\cdot)}(\mathbb{R}^n)$ as above, is a particular case of \cite[Proposition 5.3 and Theorem 5.5]{sawa}. Theorem \ref{Riesz estimates} recovers \cite[Remark 8 and Theorem 11]{Ro-Urc}. We observe that \cite[Theorem 11]{Ro-Urc} was obtained by means of the molecular decomposition for $H^{p(\cdot)}(\mathbb{R}^n)$ established in \cite{Nakai}. Proposition \ref{vector-valued fract} contains \cite[Lemma 5]{Ro-Urc} as a special case.

\

{\bf Mixed-norm Hardy spaces.} Fix $\vec{p} := (p_1, ..., p_n) \in (0, \infty)^n$. The mixed-norm Lebesgue space 
$L^{\vec{p}}(\mathbb{R}^n)$ is defined as the set of all measurable functions $f$ such that
\[
\| f \|_{\vec{p}} := \left( \int_{\mathbb{R}} \cdot \cdot \cdot 
\left[ \int_{\mathbb{R}} |f(x_1, ..., x_n)|^{p_1} dx_1 \right]^{p_2/p_1} \cdot \cdot \cdot dx_n  \right)^{1/p_n} < \infty.
\]
These spaces were introduced and studied by A. Benedek and R. Panzone in \cite{Benedek}. It is easy to check that the pair 
$(L^{\vec{p}}(\mathbb{R}^n), \| \cdot \|_{\vec{p}})$ is a ball quasi-Banach function space, and its quasi-norm 
$\| \cdot \|_{\vec{p}}$ is \textit{absolutely continuous}. For $\vec{p} \in (1, \infty)^n$, the space $L^{\vec{p}}(\mathbb{R}^n)$ is a ball Banach function space with $(L^{\vec{p}}(\mathbb{R}^n))' = L^{\vec{p'}}(\mathbb{R}^n)$, where $\vec{p'}$ denotes the $n$-tuple whose components are the H\"older conjugates of the corresponding components of $\vec{p}$. In this case we write 
$\frac{1}{\vec{p}} + \frac{1}{\vec{p'}} = \vec{1}$ (see \cite[Theorem 1.a)]{Benedek}). In \cite[Remark 7.21]{Zhang}, it was pointed out that these spaces need not be Banach function spaces.

Given $\vec{p} \in (0, \infty)^n$, we define
\[
p_{-} := \min \{ p_1, ..., p_n \}, \,\,\,\,\,\, p_{+} := \max \{ p_1, ..., p_n \}.
\]
For $0 \leq \alpha < n$ and $\vec{p} \in (0, \infty)^n$ with $0 < p_{+} < \frac{n}{\alpha}$, we set
\begin{equation} \label{q}
\vec{q} := (\frac{np_1}{n - \alpha p_1}, ..., \frac{np_n}{n - \alpha p_n}),
\end{equation}
$X:=L^{\vec{p}}(\mathbb{R}^n)$ and $Y := L^{\vec{q}}(\mathbb{R}^n)$. It is clear that if $\alpha = 0$, then $X=Y$. 
By \cite[p. 302, lines 14-19]{Benedek}, for any $s \in (0, p_{-})$, the space $L^{\vec{p}}(\mathbb{R}^n)$ is \textit{strictly $s$-convex}, with $(L^{\vec{p}})^{1/s} = L^{\vec{p}/s}$. Now, for any $p_0 \in (0, \min\{ \frac{n}{n+\alpha}, p_{-}\})$, we put 
$\frac{1}{q_0}:= \frac{1}{p_0} - \frac{\alpha}{n}$, then $0 < p_0 \leq q_0 \leq 1$, $q_0 \in (0, q_{-})$, $X^{1/p_0}$ and $Y^{1/q_0}$ are ball Banach function spaces such that $(Y^{1/q_0})' = ((X^{1/p_0})')^{p_0/q_0}$, and $M$ is bounded on $(Y^{1/q_0})'$ 
(see \cite[Lemma 3.5]{Huang}). For $X:=L^{\vec{p}}(\mathbb{R}^n)$, the assumption (\ref{A1}) holds true for any $s \in (0,1]$ and 
$\theta \in (0, \min\{s, p_{-}\})$ (apply \cite[Lemma 3.7]{Huang} on $L^{\vec{p}/\theta}$, with $u=s/\theta$ and $|f_j|^{\theta}$ instead 
of $f_j$). Now, by duality, \cite[Lemma 3.5]{Huang}, and Remark \ref{cond equiv A2}, the assumption (\ref{A2}) holds true for 
$r \in (0, \min\{ 1, p_{-}\})$ and $\widetilde{q} \in (\max \{1, p_{+}\}, \infty)$. Similarly, the assumption (\ref{A2}) holds 
true for $Y= L^{\vec{q}}(\mathbb{R}^n)$. Then, $H_{X}(\mathbb{R}^n) = H^{\vec{p}}(\mathbb{R}^n)$ and $H_{Y}(\mathbb{R}^n) = H^{\vec{q}}(\mathbb{R}^n)$ are the mixed-norm Hardy spaces defined in \cite{Huang}, taking $\vec{a} = \vec{1}$.

So, Theorem \ref{main thm} applies to these spaces $X$ and $Y$. In this case, Theorem \ref{Sing estimates} is a particular case of \cite[Theorems 6.4 and 6.5]{Huang}. Now, Theorem \ref{Riesz estimates} gives the $H^{\vec{p}}(\mathbb{R}^n) \to L^{\vec{q}}(\mathbb{R}^n)$ and the 
$H^{\vec{p}}(\mathbb{R}^n) \to H^{\vec{q}}(\mathbb{R}^n)$ boundedness for the Riesz potentials $I_{\alpha}$. To the best of the author's knowledge, the first result is new, whereas the second overlaps with \cite[Theorem 4.9]{Chen}. Therein, the authors assume $\vec{p} \in (0, \infty)^n$ such that 
$\sum_{i=1}^{n} 1/p_i \in (\alpha, \infty)$ and define $\vec{q}$ as in (\ref{q}). We point out that, in general, $\vec{q}$ is not well-defined. For instance, consider $n=2$, $0 < \alpha < 2$, $p_1 = 2/\alpha$ and $p_2 < 2/\alpha$, for these parameters, $\vec{q}$ is not defined 
since $2 - \alpha p_1 = 0$.

Proposition \ref{vector-valued fract} provides an off-diagonal Fefferman-Stein vector-valued inequality for the fractional maximal operator 
$M_{\alpha}$ on mixed-norm Lebesgue spaces. This result overlaps with \cite[Corollary 3.1]{Jian}. Therein, the author gives an off-diagonal vector-valued inequality for the strong maximal fractional $M_{R, \alpha}$ in the mixed variable-exponent setting.

\

{\bf Hardy-Lorentz spaces.} Given $0 < p, \, q < \infty$, the Lorentz space $L^{p, q}(\mathbb{R}^n)$ is defined as the collection of all measurable functions $f$ such that
\[
\| f \|_{L^{p, q}} := \left( \int_{0}^{\infty} (t^{1/p} f^{\ast}(t))^{q} \, \frac{dt}{t} \right)^{1/q} 
< \infty,
\]
where $f^{\ast}$, the decreasing rearrangement of $f$, is defined by setting, for any $t \in [0, \infty)$, 
$f^{\ast}(t) := \inf \{ s > 0 : |\{x : |f(x)| > s \}| \leq t \}$.
By \cite[Section 7.3]{sawa}, the pair $\left( L^{p, q}(\mathbb{R}^n), \| \cdot \|_{L^{p, q}} \right)$  is a ball quasi-Banach function space, whose quasi-norm $\| \cdot \|_{L^{p, q}}$ is \textit{absolutely continuous} (see \cite[(iii) on p. 23]{Wang}) and satisfies 
$\| |g |^r \|_{L^{p, q}} = \| g \|_{L^{pr, qr}}^r$ for all $0 < p, \, q, r < \infty$. 
For any $s \in (0, \min\{p, q \})$, by \cite[Theorem 4.6]{Bennett}, the space $(L^{p, q})^{1/s} = L^{p/s, q/s}$ is a Banach function space and so $L^{p, q}$ is \textit{strictly $s$-convex} (see \cite[Theorem 1.6]{Bennett}). When 
$1 < p, \, q < \infty$, $(L^{p, q})' = L^{p', q'}$, where $\frac{1}{p} + \frac{1}{p'} = 1$ and 
$\frac{1}{q} + \frac{1}{q'} = 1$ (see \cite[Theorem 4.7]{Bennett}).

Given $0 \leq \alpha < n$ and $0 < p, q < \frac{n}{\alpha}$, we put $\frac{1}{u} := \frac{1}{p} - \frac{\alpha}{n}$ and
$\frac{1}{v} := \frac{1}{q} - \frac{\alpha}{n}$. Now, we consider $X:= L^{p, q}(\mathbb{R}^n)$ and 
$Y:= L^{u, v}(\mathbb{R}^n)$. It is clear that if $\alpha = 0$, then $X=Y$. For any $p_0 \in (0, \min\{ \frac{n}{n+\alpha}, p, q \})$, we put 
$\frac{1}{q_0}:= \frac{1}{p_0} - \frac{\alpha}{n}$, then $0 < p_0 \leq q_0 \leq 1$, $q_0 \in (0, \min\{u, v\})$, $X^{1/p_0}$ 
and $Y^{1/q_0}$ are ball Banach function spaces such that $M$ is bounded on $(Y^{1/q_0})'$ (see \cite[Theorem 2.3 (i)]{Curbera}). Since  
$\frac{1}{p} - \frac{1}{u}= \frac{1}{p_0} - \frac{1}{q_0}= \frac{1}{q} - \frac{1}{v}$, we have 
$(Y^{1/q_0})' = ((X^{1/p_0})')^{p_0/q_0}$. For $X:= L^{p, q}(\mathbb{R}^n)$, the assumption (\ref{A1}) holds true for 
any $s \in (0,1]$ and $\theta \in (0, \min\{ s, p, q \})$ (apply conveniently \cite[Theorem 2.3 (iii)]{Curbera}). Now, 
by duality, \cite[Theorem 2.3 (iii)]{Curbera} and Remark \ref{cond equiv A2}, the assumption (\ref{A2}) holds true for any 
$r \in (0, \min\{ 1, p, q \})$ and $\widetilde{q} \in (\max \{1, p, q\}, \infty)$. Similarly, the assumption (\ref{A2}) holds 
true for $Y= L^{u,v}(\mathbb{R}^n)$. We observe that $H_X (\mathbb{R}^n) = H^{p, q}(\mathbb{R}^n)$ and 
$H_Y (\mathbb{R}^n) = H^{u, v}(\mathbb{R}^n)$ are the Hardy-Lorentz spaces defined in \cite{Abu}.

Then, Theorem \ref{main thm} applies to $X = L^{p, q}(\mathbb{R}^n)$ and $Y = L^{u, v}(\mathbb{R}^n)$, where $0 < p, q < \frac{n}{\alpha}$, 
$\frac{1}{u} = \frac{1}{p} - \frac{\alpha}{n}$, $\frac{1}{v} = \frac{1}{q} - \frac{\alpha}{n}$ and $0 \leq \alpha < n$. In this case, Theorem \ref{Sing estimates} overlaps with some results obtained in \cite[see p. 293]{Abu}. Now, Theorem \ref{Riesz estimates} overlaps with the results of \cite[Theorem 5.5]{Victor} (see also \cite[Theorem 4.18]{Bennett}). Proposition \ref{vector-valued fract} provides an off-diagonal Fefferman-Stein vector-valued inequality for the fractional maximal $M_{\alpha}$ on these Lorentz spaces. Recently, an inequality of this kind for $M_{\alpha}$ was established in \cite[Corollary 4]{Hao}, where the admissible range of the second Lorentz exponent differs from the one considered here.

\

{\bf Acknowledgements.} I am grateful to the referee for the careful reading of the paper and for the useful comments and suggestions which helped me to improve the original manuscript.

Pablo Rocha, Instituto de Matem\'atica (INMABB), Departamento de Matem\'atica, Universidad Nacional del Sur (UNS)-CONICET, Bah\'ia Blanca, Argentina. \\
{\it e-mail:} pablo.rocha@uns.edu.ar

\end{document}